\documentclass{amsart}

\usepackage{amsmath,amssymb,amsfonts}
\usepackage{mathrsfs,latexsym,amsthm,enumerate}
\usepackage{tikz}
\usepackage{amscd}
\usepackage[all]{xy}
\newtheorem{lemma}{Lemma}[section]
\newtheorem{proposition}[lemma]{Proposition}
\newtheorem{corollary}[lemma]{Corollary}

\newtheorem{remark}[lemma]{Remark}
\newtheorem{remarks}[lemma]{Remarks}
\newtheorem{theorem}[lemma]{Theorem}
\newtheorem{example}[lemma]{Example}

\begin{document}

\title[Groups and duality]{Subgroups of the group of homeomorphisms of the Cantor space
and a duality between a class of inverse monoids and a class of Hausdorff \'etale groupoids}

\author{Mark V. Lawson}

\address{Department of Mathematics
and the
Maxwell Institute for Mathematical Sciences,
Heriot-Watt University,
Riccarton,
Edinburgh~EH14~4AS,
United Kingdom}  
\email{M.V.Lawson@hw.ac.uk }
\thanks{The author was partially supported by an EPSRC grant  (EP/I033203/1).
He would also like to thank Collin Bleak (St. Andrews), Johannes Kellendonk (Lyon) and Daniel Lenz (Jena) for discussions on aspects of this paper.}

\begin{abstract} Under non-commutative Stone duality, there is a correspondence 
between second countable Hausdorff \'etale groupoids which have a Cantor space of identities
and
what we call Tarski inverse monoids: that is,  countable Boolean inverse $\wedge$-monoids with semilattices of idempotents which are countable and atomless.
Tarski inverse monoids are therefore the algebraic counterparts of the \'etale groupoids studied by Matui and provide a natural setting for many of his calculations.
Under this duality, we prove that natural properties of the \'etale groupoid correspond to natural algebraic properties of the Tarski inverse monoid:
effective groupoids correspond to fundamental Tarski inverse monoids
and
minimal groupoids correspond to $0$-simplifying Tarski inverse monoids.
Particularly interesting are the principal groupoids which correspond to Tarski inverse monoids 
where every element is a finite join of infinitesimals and idempotents.
Here an infinitesimal is simply a non-zero element with square zero.
The groups of units of fundamental Tarski inverse monoids  generalize the finite symmetric groups and include amongst their number the Thompson groups $G_{n,1}$ as well as the groups of units of AF inverse monoids,
Krieger's ample groups being examples.
\end{abstract}



\maketitle

\section{Introduction}

The theory developed in this paper provides a setting for understanding the connections between three classes of structure: 
groups, particularly the classical Thompson groups and their relatives, inverse monoids, and \'etale groupoids.
Of these, the ones pivotal to our theory are the inverse monoids.

Such semigroups were introduced as the abstract counterparts of pseudogroups of transformations \cite{Lawson98}.
They should be viewed as encoding information about {\em partial} symmetries much as groups encode information about {\em global} symmetries.
It is folklore that pseudogroups are also closely related to \'etale (topological) groupoids.
Indeed, the groupoid of germs of a pseudogroup is a standard construction\footnote{This groupoid has usually been the structure studied rather than the pseudogroup itself.}
whereas from an \'etale groupoid one may construct a pseudogroup by taking local bisections. 
It is curious, however,  that the exact mathematical nature of this connection has not been explored until comparatively recently.
The catalyst for doing so was Paterson's monograph \cite{Paterson} and this led, initially, to two different approaches:
that of Resende  \cite{Res1,Res2}, which explicitly generalized classical frame theory \cite{PJ}, 
and that of Lawson and Lenz \cite{LMS,Lawson10a,Lawson10b,Lawson12,LL,LL2},
which developed more functional-theoretic ideas to be found in \cite{Lenz, Kell1,Kell2}.
These two approaches are, as might be expected, complementary \cite{KL}.
As a result, the connection between inverse semigroups and \'etale groupoids is now thoroughly understood and goes under the rubric of {\em non-commutative Stone duality}.
The goal of this paper can now be precisely stated: we shall exploit this non-commutative Stone duality to obtain a deeper understanding of recent work by Matui \cite{Matui12,Matui13}.

Some of the themes of this paper such as the concept of effectiveness for \'etale groupoids and notions of simplicity are discussed in a variety of papers \cite{BCFS, Exel3, FHS, Renault2008, S2}
in relation to pseudogroups and algebras associated to \'etale groupoids.
We, however, locate these concepts in the relationship between Boolean inverse semigroups and their associated \'etale groupoids.

\section{Basic definitions}

Order-theoretic concepts permeate this paper.
Let $P$ be a poset.
If $X \subseteq P$, 
define the {\em upward closure} of $X$ by
$$X^{\uparrow} = \{ y \in P \colon \exists x \in X \mbox{ such that } x \leq y \}$$
and the {\em downward closure} of $X$ by
$$X^{\downarrow} = \{ y \in P \colon \exists x \in X \mbox{ such that } y \leq x \}.$$
We write $x^{\downarrow}$ rather than $\{x\}^{\downarrow}$.
If $X = X^{\downarrow}$ we say that $X$ is an {\em order ideal},
whereas if  $X = X^{\uparrow}$ 
we say that $X$ is {\em closed upwards}.

Recall that a {\em monoid} is a semigroup with an identity.
In a monoid, a {\em unit} is simply an invertible element and is {\em non-trivial} if it is not the identity.
The group of units of the monoid $S$ is denoted by $\mathsf{U}(S)$.
An {\em involution} is a unit $g \neq 1$ such that $g^{2} = 1$.

Our reference for inverse semigroups is \cite{Lawson98} but I will review the main definitions needed here.
An {\em inverse semigroup} is a semigroup $S$ in which for each element $s \in S$ there is a unique element $s^{-1} \in S$ such that $s = ss^{-1}s$ and $s^{-1} = s^{-1}ss^{-1}$.
The element $s^{-1}s$ is called the {\em domain idempotent} and $ss^{-1}$ is called the {\em range idempotent}.
We often write $\mathbf{d}(s) = s^{-1}s$ and $\mathbf{r}(s) = ss^{-1}$ and also $s^{-1}s \stackrel{s}{\longrightarrow} ss^{-1}$.
If $s,t \in S$ are such that $\mathbf{d}(s) = \mathbf{r}(t)$ then we say that $st$ is a {\em restricted product}.
Observe that in this case $\mathbf{d}(st) = \mathbf{d}(t)$ and $\mathbf{r}(st) = \mathbf{r}(s)$.
Our inverse semigroups will always be assumed to have a zero and, in addition, this paper deals only with monoids.
The key example of an inverse semigroup is the semigroup of all partial bijections on a set $X$, denoted by $I(X)$,  called the {\em symmetric inverse monoid} on the set $X$.
If $S$ is any inverse semigroup and $e \in S$ any idempotent, then the subset $eSe$ is an inverse subsemigroup that is also a monoid
with identity $e$. 
For this reason, it is called a {\em local monoid}.\footnote{The usual terminology is {\em local submonoid} but ours seems preferable. In ring theory, they would be called {\em corners}.}
The groups of units of the local monoids are called {\em local groups}.
The set of idempotents of an inverse semigroup is closed under multiplication and commutative.
It follows that the set of idempotents is partially ordered when we define $e \leq f$ whenever $e = ef (= fe)$ where $e$ and $f$.
In this way, the set of idempotents of $S$, denoted by $E(S)$, becomes a meet-semilattice when we define $e \wedge f = ef$.
For this reason, the set of idempotents of an inverse semigroup $S$ is often termed its {\em semilattice of idempotents}.
Additionally, if $A \subseteq S$ then $E(A) = A \cap E(S)$.
Let $S$ be an inverse semigroup.
An {\em inverse subsemigroup} of $S$ is a subsemigroup closed under inverses; 
it is said to be {\em wide} if it also contains all the idempotents of $S$.
Green's relations,  which play a major r\^ole in semigroup theory, are less apparent in inverse semigroup theory but we define two that will be used below.
If $e$ and $f$ are idempotents, define $e \, \mathscr{D} \, f$ if $e \stackrel{s}{\rightarrow} f$ for some element $s \in S$.
Define $e \, \mathscr{J} \, f$ if $SeS = SfS$.

Inverse semigroups are algebraic structures that come equipped with an algebraically defined partial order which is the abstract version of the restriction order on partial bijections.
In an inverse semigroup, define $s \leq t$ if $s = ts^{-1}s$.
Despite appearances, this definition is not biased to one side: it can be checked that $s = ss^{-1}t$
is an equivalent definition.
This relation is a partial order, called the {\em natural partial order}, and with respect to this order the inverse semigroup is partially ordered.
In addition, it is important to observe that $s \leq t$ implies that $s^{-1} \leq t^{-1}$.
Restricted to the semilattice of idempotents, it is just the order we defined on the idempotents above.
In addition, the set of idempotents is itself an order ideal.
An element of an inverse semigroup is said to be an {\em atom} if it is non-zero and the only element strictly below it is zero.

An inverse monoid $S$ with the property that for each $a \in S$ there is a unit $g$ such that $a \leq g$ is said to be {\em factorizable}.
Symmetric inverse monoids are factorizable precisely when they are finite.
Further discussion of the applications of this concept may be found in \cite{Lawson98}.
The proof of the following is routine.

\begin{lemma}\label{lem:factorizable} 
Let $S$ be an inverse monoid, and let $G$ be a subgroup of $\mathsf{U}(S)$.
Then $G^{\downarrow}$ is an inverse submonoid of $S$ 
and a factorizable inverse monoid with group of units $G$.
\end{lemma}

Although the natural partial order has always played an important r\^ole in the theory, comparatively little was done to study
inverse semigroups with what might be called lattice-like properties.
A notable exception was Leech's work \cite{Leech} on inverse semigroups in which every pair of elements has a meet, the so-called {\em inverse $\wedge$-semigroups}.
Such semigroups can be characterized using a device of Leech \cite{Leech}.
Let $S$ be an inverse monoid.
A function $\phi \colon S \rightarrow E(S)$ is called a {\em fixed-point operator} if it satisfies the following two conditions:
\begin{description}

\item[{\rm (FPO1)}]  $s \geq \phi (s)$.

\item[{\rm (FPO2)}]  If $s \geq e$ where $e$ is any idempotent then $\phi (s) \geq e$.

\end{description}

The following are proved in \cite{Leech} or are immediate from the definition.

\begin{proposition}\label{prop:leech} \mbox{}
\begin{enumerate}
\item An inverse monoid $S$ is an inverse $\wedge$-monoid if and only if it possesses a fixed-point operator.
\item If $\phi$ is the fixed-point operator then $a \wedge b= \phi (ab^{-1})b$.
\item In an inverse $\wedge$-monoid, we have that $\phi (s) = s \wedge 1$.
\item $\phi (s) = s \phi (s) = \phi (s)s$.
\item $\phi (s) \leq s^{-1}s,ss^{-1}$.
\end{enumerate}
\end{proposition}

\begin{remark} 
{\em 
Part (3) of Proposition~\ref{prop:leech} implies that in an inverse $\wedge$-monoid the fixed-point operator is uniquely defined.
We can therefore refer to {\em the} fixed-point operator.
It plays a key r\^ole in this paper.}
\end{remark}

In contrast to meets, not all pairs of elements in an inverse semigroup are eligible to have a join.
The {\em compatibility relation}, denoted by $\sim$, is defined by
$a \sim b$ if and only if $ab^{-1}$ and $a^{-1}b$ are both idempotents.
Observe that if $a,b \leq c$ then $a \sim b$.
Thus being compatible is a necessary condition for a pair of elements to have a join.
There is also a refinement of the compatability relation that will be important.
If  $ab^{-1}$ and $a^{-1}b$ are both zero we say that $a$ and $b$ are {\em orthogonal}.
In this case, we often write $a \perp b$.
If the join of an orthogonal pair of elements exists, we shall say that it is an {\em orthogonal join}.
An inverse semigroup is said to be {\em distributive} if it has all binary joins of compatible pairs of elements and multiplication distributes over any binary joins that exist.
If $A$ is a subset of a distributive inverse monoid, we denote by $A^{\vee}$ the set of joins of non-empty, finite compatible subsets of $A$.
An inverse monoid is said to be a {\em Boolean inverse monoid} if it
is a distributive inverse monoid whose semilattice of idempotents is a Boolean algebra.
If $e$ is an element of a Boolean algebra then $\bar{e}$ denotes its complement.
The theory of Boolean inverse monoids in general appears to be of great interest \cite{KLLR, W}.

In this paper, we need a weakening of the notion of factorizability.
A distributive inverse monoid $S$ is said to be {\em piecewise factorizable} if each element $s \in S$ may 
be written in the form $s = \bigvee_{i=1}^{m} s_{i}$ where for each $s_{i}$ there is a unit $g_{i}$ such that $s_{i} \leq g_{i}$.
This may be rewritten in the following form:
$$s = \bigvee_{i=1}^{m} g_{i}e_{i} $$
where $e_{i} = \mathbf{d}(s_{i})$.
In the factorizable case, we have that $s = gs^{-1}s$, which explains our choice of terminology.
The proof of the following is straightforward.

\begin{lemma} Let $S$ be a distributive inverse monoid.
Then $S$ is piecewise factorizable if and only if $S = (\mathsf{U}(S)^{\downarrow})^{\vee}$.
\end{lemma}

Since we shall be working with inverse monoids that have binary meets and binary (compatible) joins, we need to understand how they interact.
The following lemma summarizes just how.

\begin{lemma}\label{lem:order_properties} 
Let $S$ be an inverse semigroup.
\begin{enumerate}

\item If $s \sim t$ then $s \wedge t$ exists and 
$$\mathbf{d}(s \wedge t) = \mathbf{d}(s)\mathbf{d}(t)
\mbox{ and }
\mathbf{r}(s \wedge t) = \mathbf{r}(s)\mathbf{r}(t)$$
and
$$s \wedge t = s \mathbf{d}(t) = t \mathbf{d}(s) = \mathbf{r}(s)t = \mathbf{r}(t)s.$$

\item Let $S$ be a distributive inverse semigroup.
If $a \vee b$ exists then
$$\mathbf{d}(a \vee b) = \mathbf{d}(a) \vee \mathbf{d}(b)
\mbox{ and }
\mathbf{r}(a \vee b) = \mathbf{r}(a) \vee \mathbf{r}(b).$$

\item Let $S$ be a distributive inverse semigroup.
Suppose that $\bigvee_{i=1}^{m} a_{i}$ and $c \wedge \left( \bigvee_{i=1}^{m} a_{i} \right)$ exist.
Then all the meets $c \wedge a_{i }$ exist as does the join $\bigvee_{i=1}^{m} c \wedge a_{i}$
and we have that
$$
c \wedge \left( \bigvee_{i=1}^{m} a_{i} \right)
= 
\bigvee_{i=1}^{m} c \wedge a_{i}.
$$

\item Let $S$ be a distributive inverse semigroup.
Suppose that $a$ and $b = \bigvee_{j=1}^{n} b_{j}$ are such that  all the meets $a \wedge b_{j}$ exist.
Then $a \wedge b$ exists and is equal to $\bigvee_{j} a \wedge b_{j}$.

\end{enumerate}
\end{lemma}
\begin{proof} (1) \cite[Lemma 1.4.11]{Lawson98}.

(2) \cite[Proposition 1.4.17]{Lawson98}.

(3)  We begin by proving the case where $i = 2$.
Specifically,  we assume that $a \vee b$ and $c \wedge (a \vee b)$ exist.
This is just the finitary version of \cite{Res3} but needs proving.
We begin with two auxiliary results.

Suppose that $x \wedge y$ exists.
We prove that $x  \mathbf{d} (y) \wedge y$ exists and that $(x \wedge y)\mathbf{d} (y) = x\mathbf{d} (y) \wedge y$.
It is immediate that $(x \wedge y)\mathbf{d} (y) \leq x \mathbf{d}(y), y$.
Suppose now that $u \leq x\mathbf{d} (y), y$.
Then $u \leq x,y$ and so $u \leq x \wedge y$.
Thus $u\mathbf{d} (y) \leq (x \wedge y)\mathbf{d} (y)$.
But $u \leq y$ implies that $u = \mathbf{r} (u)y$ and so $u \mathbf{d} (y) = u$.
Hence $u \leq  (x \wedge y)\mathbf{d} (y)$.
It follows that  $(x \wedge y)\mathbf{d} (y) = x\mathbf{d} (y) \wedge y$, as claimed.

Suppose that $x \vee y$ exists.
We prove that $(x \vee y)\mathbf{d} (x) = x$.
Clearly, $x \leq (x \vee y)\mathbf{d} (x)$.
But $\mathbf{d} ((x \vee y)\mathbf{d} (x)) = \mathbf{d} (x \vee y)\mathbf{d} (x) = \mathbf{d} (x)$
since $\mathbf{d} (x) \leq \mathbf{d} (x \vee y)$.
But two elements bounded above having the same domain are equal.

We show that $c \wedge a$ exists.
Let $x \leq a,c$.
Then $x \leq a \vee b, c$ and so $x \leq c \wedge (a \vee b)$.
Thus $x \mathbf{d} (a) \leq (c \wedge (a \vee b))\mathbf{d} (a)$.
But $x\mathbf{d} (a) = x$.
It follows that $x \leq  (c \wedge (a \vee b))\mathbf{d} (a)$.
But $ (c \wedge (a \vee b))\mathbf{d} (a) \leq c,a$.
Thus $c \wedge a = (c \wedge (a \vee b))\mathbf{d} (a)$.
By symmetry, $c \wedge b$ exists.

Now $c \wedge a \leq a$ and $c \wedge b \leq b$ and so
$c \wedge a, c \wedge b \leq a \vee b$.
It follows that $c \wedge a \sim c \wedge b$ and so $(c \wedge a) \vee (c \wedge b)$ exists.

Observe that $(c \wedge a) \vee (c \wedge b) \leq a \vee b, c$.
Now let $x \leq c, a \vee b$.
From $x \leq a \vee b$ we get that $x = (a \vee b)\mathbf{d} (x)$.
From $x \leq a \vee b$ we get that $x \mathbf{d} (a) \leq (a \vee b)\mathbf{d} (a) = a$.
Thus $x\mathbf{d} (a) \leq c\mathbf{d} (a)$.
Hence $x \mathbf{d} (a) \leq c\mathbf{d} (a) \wedge a$.
Similarly, $x \mathbf{d} (b) \leq c\mathbf{d} (b) \wedge b$.
It follows that $x\mathbf{d} (a) \vee x\mathbf{d} (b) \leq (c\mathbf{d} (a) \wedge a) \vee (c \mathbf{d} (b) \wedge b)$.
We therefore get that $x \leq (c \wedge a)\mathbf{d} (a) \vee (c \wedge b)\mathbf{d} (b) \leq (c \wedge a) \vee (c \wedge b)$,
as required.

That proves the case where $i =2$.
The  general case is proved by induction.

(4)  Since $a \wedge b_{j} \leq b_{j}$ and $a \wedge b_{k} \leq b_{k}$ and $b_{j} \sim b_{k}$,
we have that  $a \wedge b_{j} \sim a \wedge b_{k}$.
Thus $c = \bigvee_{j=1}^{n} a \wedge b_{j}$ is defined. 
Clearly, $c \leq a, b_{j}$.
Let $x \leq a,b$.
Then $x = x \wedge b$.
Thus by part (3) above, we have that $x \wedge b_{j}$ exists for all $j$,
that $\bigvee_{j=1}^{n} x \wedge b_{j}$ exists and that
$x = x \wedge b = \bigvee_{j=1}^{n} x \wedge b_{j}$.
But $x \wedge b_{j} \leq a \wedge b_{j}$.
It follows that $x \leq c$, as required. 
\end{proof}

We shall need, in particular, the following corollary of part (4) above.

\begin{corollary}\label{cor:QI} Let $S$ be a distributive inverse semigroup.
Let $a = \bigvee_{i=1}^{m} a_{i}$ and $b = \bigvee_{j=1}^{n} b_{j}$ and suppose that all meets $a_{i} \wedge b_{j}$ exist.
Then $\bigvee_{i,j} a_{i} \wedge b_{j}$ exists as does $a \wedge b$ and we have that
$a \wedge b = \bigvee_{i,j} a_{i} \wedge b_{j}$.
\end{corollary} 

The two results above enable us to manipulate meets and joins in Boolean inverse $\wedge$-monoids
almost as easily as we would in a Boolean algebra.

\begin{proposition}\label{prop:building} Let $S$ be a distributive inverse monoid and let $T$ be an inverse submonoid.
\begin{enumerate} 

\item  $T^{\vee}$ is a distributive inverse monoid.

\item Suppose that $T$ is also a $\wedge$-monoid. Then $T^{\vee}$ is also a $\wedge$-monoid.

\end{enumerate}
\end{proposition}
\begin{proof} (1) It is routine to verify that $T^{\vee}$ is an inverse submonoid of $S$ that contains all joins of finite compatible subsets.

(2) This is immediate by Corollary~\ref{cor:QI}.
\end{proof}

The following lemma often enables us to work with orthogonal joins rather than arbitrary joins.

\begin{lemma}\label{lem:bar} Let $S$ be a Boolean inverse monoid.
\begin{enumerate}

\item Suppose that $a \sim b$ and $\mathbf{d}(a) \perp \mathbf{d}(b)$.
Then $\mathbf{r}(a) \perp \mathbf{r}(b)$.

\item Suppose that $s = \bigvee_{i=1}^{m}s_{i}$.
Then we can write $s = \bigvee_{j=1}^{n}t_{j}$, an orthogonal join,
where for each $j$ there exists $i_{j}$ such that $t_{j} \leq s_{i_{j}}$.

\end{enumerate}
\end{lemma}
\begin{proof} (1) Straightfoward.

(2) The argument is similar to \cite[Proposition~4.4.1]{Paterson}.
We deal first with the case $m = 2$, the base case.
Define 
$t_{1} = s_{1} \overline{\mathbf{d}(s_{2})}$,
$t_{2} = s_{1}\mathbf{d}(s_{2}) = s_{2} \mathbf{d}(s_{1})$
and
$t_{3} = s_{2} \overline{\mathbf{d}(s_{1})}$.
Then
$\mathbf{d}(t_{1}) = \mathbf{d}(s_{1}) \overline{\mathbf{d}(s_{2})}$,
$\mathbf{d}(t_{2}) = \mathbf{d}(s_{1}) \mathbf{d}(s_{2})$
and
$\mathbf{d}(t_{1}) = \overline{\mathbf{d}(s_{1})}\mathbf{d}(s_{2})$.
By part (1), we deduce that $t_{1},t_{2},t_{3}$ are pairwise orthogonal and by Lemma~\ref{lem:order_properties},
we have that $s_{1} = t_{1} \vee t_{2}$ and $s_{2} = t_{2} \vee t_{3}$.
Thus $s = t_{1} \vee t_{2} \vee t_{3}$, an orthogonal join.
For the general case, we use induction.
Let $s = (\bigvee_{j=1}^{n}t_{j}) \vee s_{2}$ where $s_{1} =  \bigvee_{ij=1}^{n}t_{j}$ is an orthogonal join.
We now apply the base case and use the fact that restrictions of orthogonal joins are orthogonal joins.
\end{proof}

The following properties of the fixed-point operator provide important motivation for later constructions.

\begin{lemma}\label{lem:fpo} Let $S$ be a Boolean inverse $\wedge$-monoid with fixed-point operator $\phi$.
\begin{enumerate}
\item If $g$ is a unit and $e$ is an idempotent then $\phi (ge) = \phi (g)e$.
\item If $s = \bigvee_{i=1}^{m} s_{i}$ then $\phi (s) =  \bigvee_{i=1}^{m} \phi (s_{i})$.
\end{enumerate}
\end{lemma}
\begin{proof} (1) Since $\phi (g) \leq g$ we have that $\phi (g)e \leq ge$.
Let $f \leq ge$ be any idempotent.
Observe that $f \leq e$.
Then $f \leq g$ and so $f \leq \phi (g)$ giving $f = ef \leq \phi (g)e$.

 (2) Clearly, $\bigvee_{i=1}^{m} \phi (s_{i}) \leq s$ and so $\bigvee_{i=1}^{m} \phi (s_{i}) \leq \phi (s)$.
Let $f \leq s$ where $f$ is an idempotent.
Then $f = \bigvee_{i=1}^{m} s_{i} \wedge f$ by Lemma~\ref{lem:order_properties}.
But $s_{i} \wedge f \leq \phi (s_{i})$.
Thus $f \leq \bigvee_{i=1}^{m} \phi (s_{i})$ and the result follows.
\end{proof}

Our focus in this paper will be on Boolean inverse $\wedge$-monoids.
A {\em morphism} between Boolean inverse $\wedge$-monoids 
is a unital semigroup homomorphism that maps zero to zero, preserves any binary compatible joins
and any binary meets.
The {\em kernel} of such a morphism is the preimage of zero.
An {\em ideal} $I$ in a semigroup $S$ is a subset that $a \in I$ and $s \in S$ implies that $as,sa \in I$.
An inverse semigroup that has no non-trivial ideals is said to be {\em $0$-simple}.
An ideal $I$ of a Boolean inverse monoid $S$ is said to be {\em $\vee$-closed} if
$a \vee b \in I$ whenever $a,b \in I$ and $a$ and $b$ are compatible.
A $\vee$-closed ideal will be called a {\em $\vee$-ideal}.
The following was observed by Ganna Kudryavtseva. 
See the comments before Theorem~4.18 of \cite{Lawson12}.
It should be stressed that homomorphisms in semigroup theory are rarely determined by their kernels.
Thus Boolean inverse $\wedge$-monoids are atypically `ring-like'.

\begin{proposition}\label{prop:kernels} Let $\theta \colon S \rightarrow T$ be a morphism of Boolean inverse $\wedge$-monoids.
Then the kernel of $\theta$ is a $\vee$-ideal, and $\theta$ is determined by its kernel. 
\end{proposition}




In the light of Proposition~\ref{prop:kernels},  we define a Boolean inverse $\wedge$-monoid to be {\em $0$-simplifying}\footnote{The term {\em simplifying} is taken from Kumjian \cite{Kumjian}.}
if it has no non-trivial $\vee$-closed ideals.
Every $0$-simple Boolean inverse $\wedge$-monoid is $0$-simplifying, but not conversely:
counterexamples are the finite symmetric inverse monoids which are $0$-simplifying but not $0$-simple.

We now recall some classical inverse semigroup theory.
Let $S$ be an inverse semigroup.
Define $s  \, \mu \, t$ if and only if $ses^{-1} = tet^{-1}$ for all idempotents $e \in E(S)$.
Thus $s$ and $t$ induce the same conjugation maps on $E(S)$.
Observe that if $s \, \mu \, t$ then $s^{-1}s = t^{-1}t$ and $ss^{-1} = tt^{-1}$.
In fact, if $s^{-1}s = t^{-1}t$ and $ss^{-1} = tt^{-1}$ then to check that  $s  \, \mu \, t$ it is enough to verify  $ses^{-1} = tet^{-1}$ for all idempotents $e \leq s^{-1}s$.
The relation $\mu$ is a congruence and it separates idempotents meaning that if $e$ and $f$ are idempotents such that  $e \, \mu \, f$ then $e = f$.
In fact, it is the unique {\em maximum idempotent-separating congruence} on an inverse semigroup \cite{Lawson98}.
For each semilattice $E$, 
we define the inverse semigroup $T_{E}$ called the {\em Munn semigroup} of $E$.
This consists of all order-isomorphisms between the principal order ideals of $E$.
Observe that its semilattice of idempotents is isomorphic to $E$.
The congruence $\mu$ induces a homomorphism from $S$ to $T_{E(S)}$.
Denote by $Z(E(S))$ the centralizer of the idempotents in $S$.
Clearly, $E(S) \subseteq Z(E(S))$.
If $E(S) = Z(E(S))$ then $S$ is said to be {\em fundamental}.

\begin{theorem}\label{them:classic} \mbox{}
\begin{enumerate}
\item An inverse semigroup is fundamental if and only if $\mu$ is equality.
\item Let $E$ be a semilattice.
The wide inverse subsemigroups of $T_{E}$ are fundamental and every fundamental 
inverse semigroup with semilattice of idempotents isomorphic to $E$
is isomorphic to such a wide inverse subsemigroup.
\end{enumerate}
\end{theorem}

The proof of the following is immediate.

\begin{lemma}\label{lem:finite_fundamental} A fundamental inverse semigroup is finite
if and only if its semilattice of idempotents is finite.
\end{lemma}

The following is well-known and the proof routine.

\begin{lemma}\label{lem:local_monoid} 
If an inverse semigroup is fundamental so too are all its local monoids.
\end{lemma}

Insight into the significance of fundamental inverse semigroups was first obtained by Wagner \cite[Theorem~5.2.10]{Lawson98}.

\begin{theorem}[Wagner's theorem]\label{thm:wagners}
Let $S$ be an inverse semigroup of partial bijections on a set, where the domains of definition of the idempotents generate a topology.
If that topology is $T_{0}$ then $S$ is fundamental, and every fundamental inverse semigroup is isomorphic to
one constructed in this way.
\end{theorem}

Thus, intuitively, fundamental inverse semigroups can be regarded as being inverse semigroups of partial homeomorphisms.

\begin{remark} {\em Wagner's theorem explains how the work of Kumjian \cite{Kumjian} is related to that of this paper.
Kumjian's {\em localizations} are countable, fundamental inverse semigroups whose semilattices of idempotents form {\em frames}.
Such inverse semigroups are completed to yield pseudogroups,
the pseudogroup completion of a localization being called an {\em affiliation}.
Our work on distributive and Boolean inverse monoids is derived from the theory of pseudogroups by studying {\em coherent} pseudogroups.
This theory is described in \cite{LL}, but it is worth noting that it can be viewed as a non-commutative generalization of the
theory to be found in the early chapters of \cite{PJ}. Our work is, in some sense, a finitary version of Kumjian's.} 
\end{remark}

\begin{proposition}\label{prop:munn-one} Let $E$ be a Boolean algebra. 
Then the Munn semigroup $T_{E}$ is a Boolean inverse monoid.
\end{proposition}
\begin{proof} The inverse semigroup $T_{E}$ is an inverse subsemigroup of the inverse semigroup $I(E)$.
The semilattice of idempotents of $T_{E}$ is isomorphic to $E$ and so is a Boolean algebra.
The natural partial order on $T_{E}$ is inclusion of partial bijections.
To prove the claim, we need to verify that $T_{E}$ has binary compatible joins and that multiplication distributes over such joins.
Let $\alpha_{1} \colon e_{1}^{\downarrow} \rightarrow f_{1}^{\downarrow}$ and $\alpha_{2} \colon e_{2}^{\downarrow} \rightarrow f_{2}^{\downarrow}$ be compatible order-isomorphisms in $T_{E}$.
By assumption $\alpha_{1}$ and $\alpha_{2}$ restricted to $(e_{1} \wedge e_{2})^{\downarrow}$ agree.
An important consequence of this assumption is the following.
Let $i \leq e_{1}$ and $j \leq e_{2}$.
If $ij = 0$ then $\alpha_{1}(i)\alpha_{2}(j) = 0$.
Define $\beta \colon (e_{1} \vee e_{2})^{\downarrow} \rightarrow (f_{1} \vee f_{2})^{\downarrow}$
by 
$$\beta (i) = \alpha_{1}(e_{1}i) \vee \alpha_{2}(e_{2}\overline{e_{1}}i).$$
It is now routine to check that $\beta$ is an order-isomorphism.
Clearly, $\alpha_{1},\alpha_{2} \subseteq \beta$.
Now let $\gamma$ be any element of $T_{E}$ such that $\alpha_{1},\alpha_{2}\subseteq \gamma$.
Let $i \leq e_{1}$ and $j \leq e_{2}$.
Then
$$\beta (i) = \alpha_{1}(e_{1}i) \vee \alpha_{2}(e_{2}\overline{e_{1}}i)
= \gamma (e_{1}i) \vee \gamma (e_{2}\overline{e_{1}}i)
= \gamma  (e_{1}i \vee e_{2}\overline{e_{1}}i )
= \gamma (i)$$
as required.
Observe also the following.
If $\alpha_{1}$ and $\alpha_{2}$ are regarded as elements of $I(E)$ then their union is $\alpha_{1} \cup \alpha_{2}$.
We have that $\alpha_{1} \cup \alpha_{2} \subseteq \alpha_{1} \vee \alpha_{2}$, 
and 
$\alpha_{1} \vee \alpha_{2}$ is the smallest element of $T_{E}$ that contains $\alpha_{1} \cup \alpha_{2}$.
To prove distributivity, let $\gamma \in T_{E}$ where $\gamma \colon i_{1}^{\downarrow} \rightarrow j_{1}^{\downarrow}$.
Then it is enough to prove that $\gamma \alpha_{1} \vee \gamma \alpha_{2} \leq \gamma \left(  \alpha_{1} \vee \alpha_{2}  \right)$.
But this follows by our observation above.
\end{proof}

Another perspective on the above result can be found in Proposition~\ref{prop:munn-again}.

Our next theorem, proved in  \cite{Lawson12}, has played an important r\^ole in motivating our work and also serves to bind together many of the definitions we have made so far.
We use the term {\em groupoid} in the sense of category theory.
A {\em local bisection} is a subset $A$ of a groupoid such that the sets $A^{-1}A$ and $AA^{-1}$ consist entirely of identities.
Observe that a finite Boolean inverse monoid is automatically a $\wedge$-monoid.

\begin{theorem}\label{them:bongo} Let $S$ be a finite Boolean inverse monoid.
\begin{enumerate}

\item There is a finite discrete groupoid $G$ such that $S$ is isomorphic to the set of all local bisections of $G$ under multiplication of subsets.

\item The fundamental such semigroups are the finite direct products of finite symmetric inverse monoids.

\item The $0$-simplifying, fundamental such semigroups are the finite symmetric inverse monoids.

\end{enumerate}
\end{theorem}

The whole of this paper rests on a very simple idea: that Boolean inverse $\wedge$-monoids should be viewed as non-commutative generalizations of Boolean algebras.
For example, part (1) of Theorem~\ref{them:bongo} shows that 
finite Boolean inverse $\wedge$-monoids are characterized in a way that directly extends the characterization 
of finite Boolean algebras: we just replace finite sets by finite discrete groupoids.
Boolean inverse $\wedge$-monoids have non-idempotent elements and so their groups of units are non-trivial
and, furthermore, these groups are interesting.
For example, finite symmetric groups are the groups of units of the finite symmetric inverse semigroups.
We shall be particularly interested in those Boolean inverse $\wedge$-monoids whose semilattices of idempotents are a specific Boolean algebra.
It is a theorem of Tarski, proved via a back-and-forth argument, that any two countable atomless Boolean algebras are isomorphic \cite{GH}.
Accordingly, we call the unique countable atomless Boolean algebra the {\em Tarski algebra}.\footnote{This is not standard but it is convenient.}
See Proposition~\ref{prop:blop} for why Tarski algebras naturally generalize finite Boolean algebras.

We now make a definition which is directly motivated by the finite case, as we shall see.
A non-zero element $s \in S$ is called an {\em infinitesimal} if $s^{2} = 0$.
Infinitesimals play an important r\^ole in constructing units as the following lemma demonstrates.

\begin{lemma}\label{lem:spooks} Let $S$ be a Boolean inverse $\wedge$-monoid.
\begin{enumerate}

\item Let $s \in S$ such that $s^{-1}s = ss^{-1}$.
Put $e = s^{-1}s$.
Then $g = s \vee \bar{e}$ is a unit.

\item  $a^{2} = 0$ if, and only if, $a^{-1}a \perp aa^{-1}$  if, and only if, $a \perp a^{-1}$.

\item If $a^{2} = 0$ and $e = \overline{a^{-1}a} \,   \overline{aa^{-1}}$ then
$$u = a^{-1} \vee a \vee e$$
is a non-trivial involution lying above $a$.

\item Each unit in a local monoid of $S$ lies beneath a unit of $S$.

\end{enumerate}
\end{lemma}
\begin{proof} (1) Simply observe that $g^{-1}g = gg^{-1} = s^{-1}s \vee \overline{s^{-1}s} = 1$.

(2) If $a^{2} = 0$ then $a^{-1}aaa^{-1} = 0$ and so $a^{-1}a \perp aa^{-1}$.
If  $a^{-1}a \perp aa^{-1}$ then $a^{-1}aaa^{-1} = 0$ and so $a^{2} = 0$.
The equivalence of $a^{-1}a \perp aa^{-1}$  with $a \perp a^{-1}$ is immediate.

(3) The elements $a$ and $a^{-1}$ are orthogonal.
Put $s = a \vee a^{-1}$.
Then $s^{-1}s = ss^{-1}$.
Now apply part (1).
It is straightforward to check that it is an involution.

(4) This is immediate by (1) above.

\end{proof}







\begin{example} {\em Let $I_{n}$ be a finite symmetric inverse monoid on $n$ letters.
Examples of infinitesimal elements are those partial bijections of the form $x \mapsto y$ where $x,y \in X$ and $x \neq y$.
The group elements associated with these, as constructed in Lemma~\ref{lem:spooks}, are precisely the transpositions $(xy)$
which play such an important r\^ole in the theory of finite symmetric groups.}
\end{example}

We now introduce a new class of inverse semigroups.\\

\noindent
{\bf Definition.} A {\em Tarski inverse monoid} is a countable Boolean inverse $\wedge$-monoid whose semilattice of idempotents forms a Tarski algebra.\\

This definition was motivated by the work of Matui \cite{Matui12,Matui13} who in turn was  motivated by questions coming from topological dynamics.
With these definitions in place, we can now state the first main theorem we shall prove in this paper.
We shall define minimal groupoids and effective groupoids later.

\begin{theorem}\label{thm:ONE} Under non-commutative Stone duality, Tarski inverse monoids correspond to second countable Hausdorff Boolean groupoids with unit space the Cantor space,
with the  group of units of the inverse monoid corresponding to the group of compact-open bisections of the groupoid.
Under this correspondence, 
\begin{enumerate}
\item The $0$-simplifying monoids correspond to minimal groupoids.
\item The fundamental monoids correspond to effective groupoids.
\end{enumerate}
\end{theorem}

The class of fundamental Tarski inverse monoids is therefore of particular interest.

We now define a class of groups.
A {\em Stone space} $\mathscr{S}$  is any Hausdorff, compact space with a basis of clopen subsets.
Denote by  $\mbox{Homeo}(\mathscr{S})$  the group of homeomorphisms $\mathscr{S}$.
We shall be interested in certain kinds of subgroups of $\mbox{Homeo}(\mathscr{S})$.
If $\alpha \in \mbox{Homeo}(\mathscr{S})$, denote by $\mbox{supp} (\alpha)$ the closure in $\mathscr{S}$
of the set $\{x \in \mathscr{S} \colon \alpha (x) \neq x\}$.
By a {\em Stone group}, we mean a subgroup $G$ of  $\mbox{Homeo}(\mathscr{S})$
satisfying the following two conditions.
\begin{description}
\item[{\rm (SG1)}] For each $\alpha \in G$ the set $\mbox{supp}(\alpha)$ is clopen.

\item[{\rm (SG3)}] Let $\{e_{1}, \ldots, e_{n}\}$ be a finite partition of $\mathscr{S}$ by clopen sets and let $g_{1}, \ldots, g_{n}$ be a finite subset of $G$ such that
 $\{g_{1}e_{1}, \ldots, g_{n}e_{n}\}$ is a partition of $\mathscr{S}$ also by clopen sets. Then the union of the partial bijections $g_{1} \mid e_{1}, \ldots, g_{n} \mid e_{n}$ is an element of $G$.
We call this property {\em fullness} and term {\em full } those subgroups of $\mbox{Homeo}(\mathscr{S})$ that satisfy this property.
\end{description}
A {\em Cantor group} is a countable Stone group where the Stone space is the Cantor space.

Let $G$ be a discrete group acting on the space $X$ by homeomorphisms.
We say that the action is {\em minimal} if $X$ has no proper closed $G$-invariant subset or, equivalently, if every $G$-orbit is dense.

\begin{theorem}\label{thm:TWO} \mbox{}
\begin{enumerate}

\item The group of units of a fundamental Tarski inverse monoid is a Cantor group, and each Cantor group occurs as the group of units of some fundamental Tarski inverse monoid.

\item  The group of units of a $0$-simplifying fundamental Tarski inverse monoid is a minimal Cantor group, and each minimal Cantor group occurs as the group
of units of a $0$-simplifying fundamental Tarski inverse monoid in essentially one way.

\end{enumerate}
\end{theorem}

The following theorem is just a combination of the above two, but demonstrates how the three classes of structures mentioned at the beginning of the Introduction are related.

\begin{theorem}\label{thm:THREE} The following three classes of structure are equivalent (in some precise sense).
\begin{enumerate}
\item Full subgroups of the group of homeomorphisms of the Cantor space which act minimally and in which each element has clopen support.
\item $0$-simplifying, fundamental Tarski inverse monoids.
\item Minimal, effective, second countable Hausdorff \'etale topological groupoids whose space of identities is the Cantor space.
\end{enumerate}
\end{theorem}

\section{A non-commutative dictionary}

The goal of this section is to describe how non-commutative Stone duality enables us to construct
a dictionary between Boolean inverse $\wedge$-monoids and a class of \'etale groupoids.
This provides the setting for Theorem~\ref{thm:ONE}.
We refer to \cite{Lawson10b, Lawson12} for any proofs but we adopt the na\"{\i}ve approach and directly generalize the classical Stone duality
that exists between Boolean algebras on the one hand and Stone spaces on the other.
The space associated with a Boolean algebra under this duality is called its {\em asscociated Stone space}.
In particular, the Stone space of the Tarski algebra is the Cantor space.
Let $G$ be a groupoid.
We denote its set of identities by $G_{o}$ and the domain and range maps by $\mathbf{d}$ and $\mathbf{r}$, respectively.
It is said to be a {\em topological groupoid} if the groupoid multiplication and the maps $\mathbf{d}$, $\mathbf{r}$ together with inversion are continuous,
and it is {\em \'etale} if the maps $\mathbf{d}$ and $\mathbf{r}$ are local homeomorphisms.\footnote{We shall write simply  \'etale groupoid rather than \'etale topological groupoid.}
It is the \'etale property that is crucial for the connections between \'etale groupoids and semigroups
since it implies that the open subsets of $G$ form a monoid under multiplication of subsets \cite[Chapter 1]{Res1}.
We shall be interested in \'etale groupoids where we impose additional conditions on the space of identities.\\

\noindent
{\bf Definition.} An \'etale groupoid is called {\em Boolean} if its space of identities is a Stone space.\\

We shall describe the duality that exists between Boolean inverse $\wedge$-monoids and Hausdorff Boolean groupoids.
Let $S$ be a Boolean $\wedge$-monoid.
A {\em filter} in $S$ is a subset $A \subseteq S$ which is closed under finite meets and closed upwards.
It is said to be {\em proper} if $0 \notin A$.
A maximal proper filter is called an {\em ultrafilter}.
Ultrafilters may be characterized amongst proper filters as follows.
If $A$ is a filter and $s \in S$ we write  $s \wedge A \neq 0$ to mean $s \wedge a \neq 0$ for all $a \in A$.
The following is proved as \cite[Lemma 12.3]{Exel1}.

\begin{lemma}\label{lem:exel} Let $S$ be a Boolean inverse $\wedge$-monoid.
A proper filter $A$ is an ultrafilter if and only if $s \wedge A \neq 0$ implies that $ s \in A$.
\end{lemma}

A proper filter $A$ is said to be {\em prime} if $a \vee b \in A$ implies that $a \in A$ or $b \in A$.
The following is proved as \cite[Lemma 3.20]{LL}.

\begin{lemma}\label{lem:molina} Let $S$ be a Boolean inverse $\wedge$-monoid.
A proper filter is prime if and only if it is an ultrafilter.
\end{lemma}

Let $S$ be a Boolean inverse $\wedge$-monoid
and 
denote the set of all ultrafilters of $S$ by $\mathsf{G}(S)$.
The key feature of the set of ultrafilters of such a monoid is that they support a partially defined binary operation defined as follows.
If $A$ is an ultrafilter, define
$$
\mathbf{d}(A) = (A^{-1}A)^{\uparrow}
\mbox{ and }
\mathbf{r}(A) = (AA^{-1})^{\uparrow},
$$
both ultrafilters.
Define a partial binary operation $\cdot$ on  $\mathsf{G}(S)$ by
$$A \cdot B = (AB)^{\uparrow}$$
if $\mathbf{d}(A) = \mathbf{r}(B)$, and undefined otherwise.
Then this is well-defined and $(\mathsf{G}(S), \cdot)$ is a groupoid.
Those ultrafilters that are identities in this groupoid are called {\em idempotent ultrafilters}.
They are precisely the ultrafilters that are also inverse submonoids.
Denote the set of idempotent ultrafilters by $\mathsf{G}(S)_{o}$.
All of this is proved in \cite[Proposition 2.13]{Lawson10b} and in the paragraph that precedes it.
If $F \subseteq E(S)$ is an ultrafilter then $F^{\uparrow}$ is an idempotent ultrafilter in $S$ and every idempotent ultrafilter is of this form.
If $G$ is an idempotent ultrafilter in $S$ and $a \in S$ is such that $a^{-1}a \in G$ then $A = (aG)^{\uparrow}$ is an ultrafilter
where $\mathbf{d}(A) = G$, and every ultrafilter in $S$ is constructed in this way.
Denote by $V_{a}$ the set of all ultrafilters in $S$ that contain the element $a$.
The following is proved as \cite[Lemmas 2.5,  2.21]{Lawson10b}.

\begin{lemma}\label{lem:pele} Let $S$ be a Boolean inverse $\wedge$-monoid.
\begin{enumerate}
 
\item $V_{a} \subseteq V_{b}$ if and only if $a \leq b$.

\item $V_{a}V_{b} = V_{ab}$.

\item $V_{a} \cap V_{b} = V_{a \wedge b}$.

\item If $a \vee b$ exists then $V_{a} \cup V_{b} = V_{a \vee b}$.

\item $V_{a}$ consists only of idempotent ultrafilters if and only if $a$ is an idempotent.

\end{enumerate}
\end{lemma}

Put $\tau = \{V_{a} \colon a \in S\}$.
Then $\tau$ is the basis for a topology on $\mathsf{G}(S)$
with respect to which it is a Hausdorff \'etale groupoid whose space of identities is a Boolean space.
In fact, the space of identities is homeomorphic to the Stone space of the Boolean algebra of idempotents of $S$.
Observe that if $S$ is countable then $\tau$ is a countable basis for the space and so as a topological space $\mathsf{G}(S)$ is second countable.
Thus from each Boolean inverse $\wedge$-monoid we may construct a Hausdorff Boolean groupoid using ultrafilters.
If $G$ is a Hausdorff Boolean groupoid, denote by $\mathsf{KB}(G)$ the set of all compact-open local bisections of $G$.
This is a Boolean inverse $\wedge$-monoid under multiplication of subsets.
Suppose that the Hausdorff Boolean groupoid $G$ is second countable with countable basis $\sigma$.
Then each compact-open local bisection is a union of elements of $\sigma$ and,
since compact, this union can be chosen to be finite.
It follows that there are only countably many compact-open local bisections and so  $\mathsf{KB}(G)$ is countable.
We may now state the theorem that provides the setting for this paper.
The first part was proved as the main theorem of \cite{Lawson10b} 
and the second follows from the fact that the compact-open local bisections of the \'etale groupoid form a basis for its topology.

\begin{theorem}[Non-commutative Stone Duality]\label{them:duality} For suitable definitions of morphisms,
the category of Boolean inverse $\wedge$-monoids is dually equivalent to the category of Hausdorff Boolean groupoids
with respect to the functors $S \mapsto \mathsf{G}(S)$ and $G \mapsto \mathsf{KB}(G)$.
Under this duality, countable Boolean inverse $\wedge$-monoids correspond to second countable Hausdorff Boolean groupoids.
\end{theorem}

\begin{remarks} \mbox{}
{\em 

\begin{enumerate}
\item We have deliberately sidestepped the issue of just what morphisms are needed because they play no r\^ole in this paper.
A thorough discussion of morphisms may be found in \cite{KL}.

\item It is worth observing that in the above theorem, the Hausdorffness of the Boolean groupoid is a direct consequence of the fact that the 
inverse monoid is a $\wedge$-monoid. This is discussed in \cite{LL}. But in this paper all groupoids will be Hausdorff. 
\end{enumerate}
}
\end{remarks}

Let $S$ be a Boolean inverse $\wedge$-monoid.
Denote by $\mathsf{X}(E(S))$ the Stone space of $E(S)$.
It is homeomorphic to the space of identities of  $\mathsf{G}(S)$ and we call it the {\em structure space} of $S$. 
If $e \in E(S)$, we denote by $U_{e}$ the set of all ultrafilters {\em in} $E(S)$ that contain $e$.

\section{Proof of Theorem~\ref{thm:ONE}}

In Section~4.1 and Section~4.2 below, we prove, respectively, the two cases of Theorem~\ref{thm:ONE}.
We then study two refinements of these results.

\subsection{The $0$-simplifying case}

As we have already seen, 
morphisms between Boolean inverse $\wedge$-monoids behave like morphisms between Boolean algebras in that they are determined by their kernels.
Such kernels are the $\vee$-ideals, and the monoids with only trivial such ideals are $0$-simplifying.
Thus $0$-simplifying is a natural notion of simplicity for Boolean inverse $\wedge$-monoids.
The following relation was introduced in \cite{Lenz} and will be important in handling $0$-simplifying monoids.
Let $e$ and $f$ be two non-zero idempotents in $S$.
Define $e \preceq f$ if and only if there exists a set of  elements $X = \{x_{1}, \ldots, x_{m}\}$ such that
$e = \bigvee_{i=1}^{m}  \mathbf{d}(x_{i})$
and 
$\mathbf{r}(x_{i}) \leq f$ for $1 \leq i \leq m$. 
We can write this formally as $e = \bigvee \mathbf{d}(X)$ and $\bigvee \mathbf{r}(X) \leq f$.
We say that $X$ is a {\em pencil} from $e$ to $f$.

\begin{lemma}\label{lem:preceq} \mbox{}
\begin{enumerate}

\item The relation $\preceq$ is a preorder on the set of idempotents.

\item The relation $\preceq$ is preserved by morphisms.

\item Let $I$ be a non-zero $\vee$-ideal.
If $e \in I$ and $f \preceq e$ then $f \in I$. 

\end{enumerate}
\end{lemma}
\begin{proof} (1) Since $e \leq e$ the relation $\preceq$ is reflexive.
We prove that it is transitive.
Let $X = \{x_{i} \}$ be a pencil from $e$ to $f$ and let $Y = \{y_{j} \}$ be a pencil from $f$ to $g$.
We prove that $YX$ is a pencil from $e$ to $g$.
Observe that 
$\mathbf{d}(y_{j}x_{i}) \leq \mathbf{d}(x_{i}) \leq e$
and
$\mathbf{r}(y_{j}x_{i}) \leq \mathbf{r}(y_{j}) \leq g$.
In addition
$$\bigvee_{i,j} \mathbf{d}(y_{j}x_{i}) 
=  
\bigvee_{i} x_{i}^{-1} \left(  \bigvee_{j} \mathbf{d}(y_{j})    \right) x_{i}
=
\bigvee_{i} \mathbf{d} (x_{i})
=
e.$$
This proves that $\preceq$ is a preorder.

(2) Suppose that $\theta \colon S \rightarrow T$ is a morphism and that $e \preceq f$ in $S$.
Then because morphisms preserve finite joins, it follows that $\theta (e) \preceq \theta (f)$ in $T$.

(3) By definition, there is a pencil $\{x_{i}\}$ where $\mathbf{r}(x_{i}) \leq e$ and $f = \bigvee_{i=1}^{n} \mathbf{d}(x_{i})$.
But $ex_{i} = x_{i}$ and so $x_{i} \in I$, since $I$ is an ideal,
and similarly $\mathbf{d}(x_{i}) \in I$ for each $i$.
We now use the fact that $I$ is a $\vee$-ideal and so $f \in I$, as required.
\end{proof}

The following lemma describes an idea that is used repeatedly in this paper.

\begin{lemma}\label{lem:useful} Let $S$ be a Boolean inverse $\wedge$-monoid.
Let $e$ and $f$ be idempotents such that $e \preceq f$ and suppose that $e \in F$ an ultrafilter in $E(S)$.
Then there is an element $a$ such that $\mathbf{d}(a) \in F$, $\mathbf{d}(a) \leq e$ and $\mathbf{r}(a) \leq f$.
\end{lemma}
\begin{proof} By definition, there is a pencil $\{x_{1}, \ldots, x_{m}\}$ from $e$ to $f$.
In particular, $e = \bigvee_{i=1}^{m} \mathbf{d}(x_{i})$.
But every ultrafilter is a prime filter by Lemma~\ref{lem:molina} and so $\mathbf{d}(x_{i}) \in F$ for some $i$.
The element $a = x_{i}$ therefore has the required properties.
\end{proof}

In the light of Lemma~\ref{lem:preceq}, 
we may define the equivalence relation $e \equiv f$ if and only if $e \preceq f$ and $f \preceq e$.
The following was proved as part of Lemma~7.8 of \cite{Lenz}.

\begin{lemma}\label{lem:toby} Let $S$ be a Boolean inverse $\wedge$-monoid.
Then $\equiv$ is the universal relation on the set of non-zero idempotents if and only if $S$ is $0$-simplifying.
\end{lemma}


The following was suggested by the first line in the proof of \cite[Theorem~6.11]{Matui12}.

\begin{proposition}\label{prop:blop} Let $S$ be a countable Boolean inverse $\wedge$-monoid.
If $S$ is $0$-simplifying then either $S$ is a Tarski inverse monoid or the semilattice of idempotents of $S$ is finite.
\end{proposition}
\begin{proof} There are two possibilities.
Suppose first that $E(S)$ contains no atoms.
Then it is clearly not finite and so, since it is countable, it is a Tarski algebra.
Thus in what follows, we suppose that $E(S)$ contains at least one atom $e$.
We prove first that if $e$ is an atom and $e \, \mathscr{D} \, f$ then $f$ is an atom.
Let $e \stackrel{a}{\rightarrow} f$.
Suppose that $i \leq f$.
Then $ia \leq a$.
Hence $\mathbf{d}(ia) \leq e$.
Since $e$ is an atom, it follows that $\mathbf{d}(ia) = e$ or  $\mathbf{d}(ia) = 0$.
If $\mathbf{d}(ia) = 0$ then $ia = 0$ and so $i = 0$.
If  $\mathbf{d}(ia) = e$ then $ia = a$ and $i = f$.
It follows that $f$ is an atom.
Now let $f$ be any non-zero idempotent.
We are assuming that the semigroup is $0$-simple and so $f \preceq e$ by Lemma~\ref{lem:toby}.
There are therefore a finite number of {\em non-zero} elements $x_{1}, \ldots, x_{m}$ of $S$ such that
$f = \bigvee_{i=1}^{m} \mathbf{d}(x_{i})$ and $\mathbf{r}(x_{i}) \leq e$.
But $e$ is an atom and so $\mathbf{r}(x_{i}) = e$ for $i = 1, \ldots, m$.
It follows that $\mathbf{d}(x_{i})$ is an atom.
We have therefore proved that each non-zero element of $E(S)$ is the join of a finite number of atoms,
and all the atoms are $\mathscr{D}$-related to $e$.
It is an immediate consequence that any atom in $E(S)$ is $\mathscr{D}$-related to $e$.
Thus all the atoms of $E(S)$ form a single $\mathscr{D}$-class.
Since the identity is an idempotent it is the join of a finite number of atoms, say $e_{1}, \ldots, e_{m}$.
Then by distributivity, every non-zero idempotent is a join of some of these $m$ idempotents.
It follows that there are exactly $m$ atoms and the Boolean algebra of idempotents is finite with $2^{m}$ elements.
\end{proof}

The following corollary to Proposition~\ref{prop:blop} is immediate by  Lemma~\ref{lem:finite_fundamental} and Theorem~\ref{them:bongo}.

\begin{corollary} 
A fundamental $0$-simplifying countable Boolean inverse $\wedge$-monoid is either a Tarski inverse monoid or a finite symmetric inverse monoid.
\end{corollary}

A groupoid $G$ is said to be {\em connected} if given any two identities $e$ and $f$ there is an element $g$ such that $e = g^{-1}g$ and $f = gg^{-1}$.
Every groupoid decomposes into a disjoint union of connected groupoids which also leads to a partition of the set of identities.
We now make explicit the relations that underly these results.
Identities $e$ and $f$ are {\em connected} if there is an element $g \in G$ such that $e \stackrel{g}{\rightarrow} f$.
This is an equivalence relation on the set $G_{o}$ and the equivalence classes are called {\em $G$-orbits}.
The $G$-orbit containing $e$ is denoted by $G(e)$.
A subset of $G_{o}$ is called an {\em invariant set} if it is a union of $G$-orbits.
Both $\emptyset$ and $G_{o}$ are invariant sets called the {\em trivial invariant sets}.
We say that two elements $g,h \in G$ are {\em connected} if their identities $g^{-1}g$ and $h^{-1}h$ are connected.
This relation is an equivalence relation whose equivalence classes are called {\em connected components}.
An {\em invariant subset} of $G$ is any union of connected components.
Lenz \cite{Lenz} remarks that the equivalence of (2) and (3) below is well-known in the \'etale case.

\begin{lemma}\label{lem:oinky} Let $G$ be an \'etale groupoid. 
Then the following are equivalent.
\begin{enumerate}

\item Every $G$-orbit is a dense subset of $G_{o}$.

\item Any non-empty invariant subset of $G_{o}$ is dense in $G_{o}$.

\item There are no non-trivial open invariant subsets of $G_{o}$.

\item There are no non-trivial open invariant subsets of $G$.

\end{enumerate}
\end{lemma}





We shall say that an \'etale groupoid is {\em minimal} if it satisfies any one of the equivalent conditions in Lemma~\ref{lem:oinky}. 
The following was sketched in \cite{Lawson12} and is a reformulation of a result in \cite{Lenz}.
It generalizes the theorem from classical Stone duality that states that the lattice of ideals of a Boolean algebra is isomorphic 
to the lattice of open subsets of its associated Stone space \cite[Theorem 33]{GH}.

\begin{theorem}[$\vee$-ideals and open invariant subsets]\label{thm:one1} Let $S$ be a Boolean inverse $\wedge$-monoid and let $\mathsf{G}(S)$ be its associated Hausdorff Boolean groupoid.
Then there is an order isomorphism between the partially ordered set of $\vee$-ideals in $S$ and the partially ordered set of open invariant subsets of $\mathsf{G}(S)$.
\end{theorem}

The following corollary is immediate by Theorem~\ref{thm:one1}.

\begin{corollary}  Let $S$ be a Boolean inverse $\wedge$-monoid. 
The groupoid $\mathsf{G}(S)$ is minimal if and only if $S$ is $0$-simplifying.
\end{corollary}

\subsection{The fundamental case}

Let $G$ be an \'etale topological groupoid.
The union of the local groups of $G$, denoted by $\mbox{Iso}(G)$, is a subgroupoid, called the {\em isotropy subgroupoid} of $G$.
We say that $G$ is {\em principal}\footnote{The significance of principal groupoids is explained in \cite[page 3]{Renault}.} if $\mbox{Iso}(G) = G_{0}$
and that it is {\em effective}\footnote{See \cite[Example~1.5]{FHS} for the explanation behind the definition of effective as well as \cite{Renault}.}
if the interior of  $\mbox{Iso}(G)$, 
denoted by  $\mbox{Iso}(G)^{\circ}$, is equal $G_{o}$.
There is a third notion to be found in \cite{Renault}: namely, that of being {\em topologically principal}.
However, in the case of Tarski inverse monoids this is equivalent to being effective.
The following lemma establishes the link betwen algebra and topology.

\begin{lemma}\label{lem:charlie} 
Let $S$ be a Boolean inverse $\wedge$-monoid, 
$\mathsf{G}(S)$ its associated groupoid
and $a \in S$.
Then 
$$a \in Z(E(S)) \Leftrightarrow V_{a} \subseteq \mbox{\rm Iso}(\mathsf{G}(S)).$$ 
\end{lemma}
\begin{proof}
Let $a \in Z(E(S))$.
We shall prove that every element of $V_{a}$ belongs to the isotropy groupoid.
Let $A \in V_{a}$.
We need to prove that  $A^{-1} \cdot A = A \cdot A^{-1}$.
We have that $A = (aA^{-1} \cdot A)^{\uparrow}$.
Now $A \cdot A^{-1} = (a A^{-1}A a^{-1})^{\uparrow}$.
Let $x \in  A \cdot A^{-1}$.
Then $aea^{-1} \leq x$ for some idempotent $e \in A^{-1} \cdot A$.
But by assumption, $a$ commutes with all idempotents.
Thus $aa^{-1}e \leq x$ and $aa^{-1} = a^{-1}a$.
Hence $a^{-1}ae \leq x$.
But $a^{-1}a,e \in A^{-1} \cdot A$ and so $a^{-1}ae \in A^{-1} \cdot A$.
It follows that $x \in A^{-1} \cdot A$.
We have therefore proved that $A \cdot A^{-1} \subseteq A^{-1} \cdot A$.
Now let $x \in   A^{-1} \cdot A$.
Then $e \leq x$ where $e \in A^{-1} \cdot A$ is an idempotent.
Clearly, $ea^{-1}a \leq x$ since $a^{-1}a \in A^{-1} \cdot A$.
But $ea^{-1}a = eaa^{-1} = aea^{-1}$.
It follows that $x \in A \cdot A^{-1}$.
We have therefore proved that $A \in \mbox{Iso}(\mathsf{G}(S))$.
But $A$ was arbitrary,
so we have proved that $V_{a} \subseteq \mbox{Iso}(\mathsf{G}(S))$.

Conversely, let 
$V_{a} \subseteq \mbox{Iso}(G)$.
We shall prove that $a \in Z(E(S))$.
Let $e \in E(S)$ be an arbitrary idempotent.
We claim that $V_{ea} = V_{ae}$.
Let $A \in V_{ea}$.
Then $ea \leq a$ and so $a \in A$.
It follows that $A \in V_{a}$.
By assumption, $A^{-1} \cdot A = A \cdot A^{-1}$.
Now $ea \in A$ implies that $ea(ea)^{-1} \in A \cdot A^{-1}$.
Thus by assumption, $ea(ea)^{-1} \in A^{-1} \cdot A$.
Hence $aeaa^{-1}e \in A$ and so $ae \in A$.
We have show that $A \in V_{ae}$.
The reverse inclusion follows by symmetry.
We now get that $ae = ea$ by Lemma~\ref{lem:pele}.
\end{proof}

The above lemma is the key result needed to prove the main theorem of this section.

\begin{theorem}\label{thm:one2} Let $S$ be a Boolean inverse $\wedge$-monoid.
Then its associated Boolean groupoid $\mathsf{G}(S)$ is effective
if and only if $S$ is fundamental.
\end{theorem}
\begin{proof} Suppose first that $\mathsf{G}(S)$ is effective.
Let $a \in Z(E(S))$.
Then by Lemma~\ref{lem:charlie}, 
$V_{a}$ is contained in the interior of the isotropy subgroupoid.
Hence $V_{a} \subseteq \mathsf{G}(S)_{o}$.
It follows that every ultrafilter containing $a$ is an idempotent ultrafilter
which by Lemma~\ref{lem:pele} implies that $a$ is an idempotent.
Thus $S$ is fundamental.
Conversely suppose that $S$ is fundamental.
Let $V_{a} \subseteq \mbox{Iso}(G)$.
Then by Lemma~\ref{lem:charlie}, we have that $a$ centralizes all idempotents.
But $S$ is fundamental and so $a$ is an idempotent.
It follows that every ultrafilter in $V_{a}$ is idempotent.
Thus the interior of the isotropy groupoid is the space of identities.
\end{proof}

\subsection{Refinement: the $0$-simple case}

Our goal is to characterize $0$-simple Boolean inverse $\wedge$-monoids amongst the $0$-simplifying ones.
The following result, proved in \cite[Proposition 3.2.10]{Lawson98}, characterizes $0$-simplicity in terms of the inverse semigroup analogue of the {\em Murray-von Neumann order}.

\begin{lemma}\label{lem:helen}
An inverse semigroup with zero is $0$-simple if and only if for any two non-zero idempotents $e$ and $f$
there exists an idempotent $i$ such that $e \, \mathscr{D} \, i \leq f$. 
\end{lemma}

The following simple result is the key to this section.

\begin{lemma}\label{lem:purely_infinite} Let $S$ be a $0$-simple Tarski inverse monoid and let $e \in S$ be any non-zero idempotent.
Then we may find a pair of elements $x,y \in S$ such that $\mathbf{d}(x) = e = \mathbf{d}(y)$,
and $\mathbf{r}(x)$ and $\mathbf{r}(y)$ are orthogonal,
and $\mathbf{r}(x) \vee \mathbf{r}(y) \leq e$.
\end{lemma}
\begin{proof} Since $S$ is atomless, there is a non-zero idempotent $f < e$.
The idempotents of $S$ form a Boolean algebra and so there is an idempotent $f'$ such that
$e = f \vee f'$ and $ff' = 0$.
By Lemma~\ref{lem:helen}, there exists an element $x$ such that
$e \stackrel{x}{\rightarrow} e_{1} \leq f$
and an element $y$ such that
$e \stackrel{y}{\rightarrow} e_{2} \leq f'$.
\end{proof}

Recalling the fact that the elements of inverse semigroups abstract partial bijections,
we might define an idempotent $e$ of an inverse semigroup to be {\em Dedekind infinite} if $e \stackrel{x}{\rightarrow} i < e$
for some element $x$.
A stronger notion is the following.
A non-zero idempotent $e$ is said to be {\em properly infinite} if  we may find a pair of elements $x,y \in S$ such that 
$e \stackrel{x}{\rightarrow} i < e$ and $e \stackrel{y}{\rightarrow} j < e$ and $i \perp j$.
An inverse monoid is said to be {\em purely infinite} if every non-zero idempotent is properly infinite.
This terminology is generalized from \cite{Matui13}.

\begin{remark} {\em The {\em polycyclic inverse monoid $P_{2}$} is the inverse monoid with zero generated by elements 
$p$ and $q$ such that $1 = p^{-1}p = q^{-1}q$ and $pp^{-1}qq^{-1} = 0$.
This inverse monoid is discussed in detail in Section~9.3 of \cite{Lawson98}.
It is, in particular, congruence-free.
Let $e$ be a properly infinite idempotent in the inverse monoid $S$.
Then there is a monoid homomorphism $P_{2} \rightarrow eSe$,
where $eSe$ is the {\em local submonoid} determined by $e$.
This homomorphism is an embedding since $P_{2}$ is congruence-free.}
\end{remark}

In the light of the above remark, we may rephrase Lemma~\ref{lem:purely_infinite} in the following terms.

\begin{corollary}\label{cor:beatrice} 
In a $0$-simple Tarski inverse monoid every non-zero idempotent is properly infinite.
In particular, each local monoid contains a copy of $P_{2}$.
\end{corollary}

This result will lead us to an exact formulation of the difference between $0$-simple and $0$-simplifying.

\begin{lemma}\label{lem:benedick} Let $S$ be a Tarski inverse monoid and let $e$ and $f$ be any two non-zero idempotents such that $e \preceq f$. 
Then we may find elements 
$x_{1}, \ldots, x_{m}$ such that
$\mathbf{r}(x_{i}) \leq f$ for $1 \leq i \leq m$ 
and 
$e = \bigvee_{i=1}^{m}  \mathbf{d}(x_{i})$
where this is an {\em orthogonal} join of idempotents.
\end{lemma}
\begin{proof} From the definition of $\preceq$ we may find such elements $y_{j}$ such that the following hold
$y_{1}, \ldots, y_{m}$ such that
$\mathbf{r}(y_{i}) \leq f$ for $1 \leq i \leq m$ 
and 
$e = \bigvee_{i=1}^{m}  \mathbf{d}(y_{i})$.
Put $e_{i} = \mathbf{d}(y_{i})$.
Define idempotents $f_{1}, \ldots, f_{n}$ as follows
$f_{1} = e_{1}$, $f_{2} = e_{2} \overline{e_{1}}$, \ldots, $f_{n} = e_{n} \overline{(e_{1} \vee \ldots \vee e_{n-1})}$.
These idempotents are pairwise orthogonal and their join  is $e$.
Observe that $f_{i} \leq e_{i}$.
Define $x_{i} = y_{i}f_{f}$.
Then $\mathbf{d}(x_{i}) = f_{i}$.
Clearly $\mathbf{r}(x_{i}) \leq f$.
\end{proof}

We now have the following result suggested by \cite[Proposition 4.11]{Matui13}.

\begin{theorem}\label{thm:leonato} Let $S$ be a Tarski inverse monoid.
Then the following are equivalent.
\begin{enumerate}

\item  $S$ is $0$-simple.

\item $S$ is $0$-simplifying and purely infinite.

\end{enumerate}
\end{theorem}
\begin{proof} (1)$\Rightarrow$(2).
Every $0$-simple semigroup is $0$-simplifying,
and we proved in Corollary~\ref{cor:beatrice}, that in a $0$-simple Tarski monoid every non-zero idempotent is properly infinite.

(2)$\Rightarrow$(1).
Let $e$ and $f$ be any two non-zero idempotents.
From the fact that the monoid is $0$-simplifying, and  Lemma~\ref{lem:benedick}, 
we may find elements $w_{1}, \ldots, w_{n}$ such that
$e = \bigvee_{i=1}^{n} \mathbf{d}(w_{i})$ is an orthogonal join
and $\mathbf{r}(w_{i}) \leq f$.
From the fact that the monoid is purely infinite, we may find elements $a$ and $b$ such that
$\mathbf{d}(a) = f = \mathbf{d}(b)$ and $\mathbf{r}(a),\mathbf{r}(b) \leq f$
and $\mathbf{r}(a)$ and $\mathbf{r}(b)$ are orthogonal.
Thus, in particular, $a^{-1}b = 0$ and $a^{-1}a = e = b^{-1}b$.
Define the elements $v_{1}, \ldots, v_{n}$ as follows:
$v_{1} = a$, $v_{2} = ba$, $v_{3} = b^{2}a$, \ldots, $v_{n} = b^{n-1}a$.
Observe that $\mathbf{d}(v_{i}) = f$ and that the $\mathbf{r}(v_{i}) \leq f$ are pairwise orthogonal.
Consider now the elements $v_{1}w_{1}, \ldots, v_{n}w_{n}$.
It is easy to check that these elements are pairwise orthogonal.
We may therefore form the join $w = \bigvee_{i=1}^{n} v_{i} w_{i}$.
Clearly, $\mathbf{d}(w) = e$ and $\mathbf{r}(w) \leq f$.
The result now follows by Lemma~\ref{lem:helen}.
\end{proof}

The following generalizes part (3) of \cite[Proposition 4.10]{Matui13}.

\begin{lemma}\label{lem:conrade} Let $S$ be a $0$-simple Tarski inverse monoid.
Let $e$ and $f$ be idempotents such that $e \neq 1$ and $f \neq 0$.
Then there is a unit $g$ such that $geg^{-1} \leq f$.
\end{lemma}
\begin{proof}
Suppose first that $f \bar{e} \neq 0$.
Since $S$ is $0$-simple, there exists $a \in S$ such that $\mathbf{d}(a) = e$ and $\mathbf{r}(a) \leq f \bar{e}$.
Clearly, $\mathbf{d}(a)$ and $\mathbf{r}(a)$ are orthogonal.
Thus $a^{2} = 0$.
By Lemma~\ref{lem:spooks}, we may  define $g = a \vee a^{-1} \vee i$, a unit, where $i = 1 \overline{(\mathbf{d}(a) \vee \mathbf{r}(a))}$.
Thus $ie = 0$.
We have that $geg^{-1} \leq f$.

Suppose now that $f \bar{e} = 0$.
Then $f < e$.
By the above result, we may find a unit $u$ such that
$ueu^{-1} \leq \bar{e}$.
Similarly, we may find a unit $v$ such that $v\bar{e}v^{-1} \leq f$.
Thus $vu e (vu)^{-1} \leq f$, as required.
\end{proof}

\begin{example}{\em 
A meet semilattice $E$ with zero is said to be {\em $0$-disjunctive} if for all non-zero $e \in E$ and $0 \neq f < e$ there exists
$0 \neq f' \leq e$ such that $f \wedge f' = 0$.
Observe that Boolean algebras are automatically $0$-disjunctive.
An inverse semigroup is said to be {\em congruence-free} if it has exactly two congruences.
It is a standard theorem that an inverse semigroup with zero is congruence-free if and only if
it is fundamental, $0$-simple and its semilattice of idempotents is $0$-disjunctive \cite{Munn}.
It follows that the fundamental, $0$-simple Tarski inverse monoids are precisely the congruence-free Tarski inverse monoids.
In \cite{Lawson07b}, we developed some ideas of  Birget \cite{Birget} to construct a family of Boolean inverse $\wedge$-monoids $C_{n}$ 
where $n \geq 2$ called the {\em Cuntz inverse monoids}.\footnote{This is a different class from the one referred to by Renault \cite{Renault} which are properly polycyclic inverse monoids.
Our Cuntz inverse monoids are direct analogues of Cuntz $C^{\ast}$-algebras.}
Their groups of units are the Thompson groups $G_{n,1}$.
The group $G_{2,1}$ is often denoted by $V$.
The construction of these inverse monoids was also given in \cite{LS}.
They are $0$-simple and fundamental and so provide concrete examples of congruence-free Tarski inverse monoids.
Matui's paper \cite{Matui13} deals with second countable Hausdorff  \'etale groupoids
which are locally compact and some of his key results deal with the case where the space of identities is a Cantor set.
In the light of our results in this section, we may apply \cite[Theorem 3.10]{Matui13}, a theorem of a type pioneered by Matti Rubin \cite{Rubin},
to deduce that $0$-simplifying, fundamental Tarski inverse monoids are isomorphic if and only if their groups of units are isomorphic.
In particular,  two congruence-free Tarski inverse monoids are isomorphic if and only if their groups of units are isomorphic.
}
\end{example}

\subsection{Refinement: the principal case}

The support of a partial homeomorphism is an important tool in Matui's work \cite{Matui12,Matui13}.
Abstract support operators on Boolean algebras are considered by Fremlin \cite{Fremlin}.
We use Leech's fixed-point operator $\phi$ of Proposition~\ref{prop:leech} to define an abstract support operator $\sigma$ on any Boolean inverse $\wedge$-monoid $S$.
Define the {\em support operator} $\sigma$ by
$$\sigma (s) = \overline{\phi (s)}s^{-1}s$$
where $s \in S$.
The idempotent $\sigma (s)$ is called the {\em support of $s$}.

\begin{lemma}\label{lem:cooper} Let $S$ be a Boolean inverse $\wedge$-monoid.
Then
$$s = \phi (s) \vee s \sigma (s)$$
is an orthogonal join, and $\phi (s \sigma (s)) = 0$.
\end{lemma}
\begin{proof} Let $s \in S$.
Then $1 = \phi (s) \vee \overline{\phi (s)}$.
Multiplying on the right by $s^{-1}s$ and observing that $\phi (s) \leq s^{-1}s$,
we get that $s^{-1}s = \phi (s) \vee \sigma (s)$.
Multiplying on the left by $s$ and observing that $s \phi (s) = \phi (s)$,
we get that $s = \phi (s) \vee s \sigma (s)$.
It is routine to check that $\phi (s) \perp s \sigma (s)$, and that $\phi (s \sigma (s)) = 0$.
\end{proof}

If $g,h \in U(S)$, define $[g,h] = ghg^{-1}h^{-1}$, the {\em commutator} of $g$ and $h$.
Part (2) below is further evidence of the interaction between the properties of the group of units of the monoid
and the properties of the monoid as a whole.

\begin{lemma}\label{lem:jerry} Let $S$ be a Boolean inverse $\wedge$-monoid.
\begin{enumerate}

\item If $a,b \in S$ are such that $\overline{\phi (a)} \,   \overline{\phi (b)} = 0$ then $ab = ba$.

\item If $g,h \in U(S)$ and $\sigma (g)\sigma (h) = 0$ then $[g,h] = 1$.

\item Let $g$ and $h$ be units. Then $\sigma (ghg^{-1}) = g \sigma (h) g^{-1}$. 

\end{enumerate}
\end{lemma}
\begin{proof} (1) From $1 =  \phi (a) \vee \overline{\phi (a)}$ and $a = 1a1$.
It quickly follows that
$$a = \overline{\phi (a)} a \overline{\phi (a)} \vee \phi (a).$$
In addition, easy calculations show that
$\overline{\phi (a)}  \leq \phi (b)$ and $\overline{\phi (b)} \leq \phi (a)$.
We calculate
$$ab = \phi (a) \phi (b) \vee \overline{\phi (b)} b \overline{\phi (b)} \vee \overline{\phi (a)} a \overline{\phi (a)}.$$
By symmetry, this is equal to $ba$.

(2) This is immediate by (1), and the fact that when $g$ is a unit $\sigma (g) = \overline{\phi (g)}$.

(3) From $\phi (h) \leq h$ we get that $g \phi (h) g^{-1} \leq ghg^{-1}$.
But  $g \phi (h) g^{-1}$ is an idempotent and so $g \phi (h) g^{-1} \leq \phi (ghg^{-1})$.
From $\phi (ghg^{-1}) \leq ghg^{-1}$ we get that $g^{-1} \phi (ghg^{-1})g \leq h$. 
But  $g^{-1} \phi (ghg^{-1})g$ is an idempotent and so $g^{-1} \phi (ghg^{-1})g \leq \phi (h)$.
It follows that $\phi (ghg^{-1}) \leq g \phi (h) g^{-1}$.
We have proved that $g \phi (h) g^{-1} = \phi (ghg^{-1})$.
The result now follows by taking complements.
\end{proof}

\begin{lemma}\label{lem:jinxy} Let $S$ be a fundamental Boolean inverse $\wedge$-monoid.
\begin{enumerate}

\item Suppose that $af = fa$ for all $f \leq \overline{\phi (a)}$.
Then $a$ is an idempotent.

\item If $a^{-1}a = aa^{-1}$ and  $af = fa$ for all $f \leq \sigma (a) $.
Then $a$ is an idempotent.

\end{enumerate}
\end{lemma}
\begin{proof} (1 ) Let $e$ be an arbitrary idempotent.
Since $1 = \overline{\phi (a)} \vee \phi (a)$ we have that $e = e \overline{\phi (a)} \vee e \phi (a)$.
Put $i = e \overline{\phi (a)} \leq \overline{\phi (a)}$ and $j = e \phi (a) \leq \phi (a)$.
By assumption, $ai = ia$.
Clearly $j \leq a$ and so $j = aj = ja$.
Thus $a$ commutes with both $i$ and $j$ and so $a$ commutes with $e$.
But $e$ was arbitrary and so $a$ commutes with every idempotent.
Since $S$ is fundamental, it follows that $a$ is an idempotent.

(2) This is immediate by (1).
\end{proof}

We give an explicit proof of the following because of its importance.
Recall that $U_{e}$ is the set of ultrafilters in $E(S)$ that contain the idempotent $e$.

\begin{lemma}\label{lem:dream}
Let $F \in U_{s^{-1}s}$.
\begin{enumerate}

\item $(sFs^{-1})^{\uparrow}$ is an ultrafilter in $S$.

\item If $sFs^{-1} \subseteq F$ then  $F$ is the only ultrafilter in $E(S)$ containing $sFs^{-1}$.

\end{enumerate}
\end{lemma} 
\begin{proof} (1) It is easy to check that  $(sFs^{-1})^{\uparrow}$ is a proper filter.
To show that it is an ultrafilter, we shall use Lemma~\ref{lem:exel}.
Suppose that $a$ is such that $a \wedge (sFs^{-1})^{\uparrow} \neq 0$.
Then we need to show that $a \in (sFs^{-1})^{\uparrow}$.
But $a \wedge (sFs^{-1})^{\uparrow} \neq 0$ if and only if  $(a \wedge ss^{-1}) \wedge (sFs^{-1})^{\uparrow} \neq 0$. 
Thus to prove that  $(sFs^{-1})^{\uparrow}$  is an ultrafilter it is enough to prove
that if $e$ is an idempotent where $e \leq ss^{-1}$ such that $e \wedge (sFs^{-1})^{\uparrow} \neq 0$ then $e \in  (sFs^{-1})^{\uparrow}$. 
To this end, 
let $e$ be an idempotent $e \leq ss^{-1}$ such that
$e(sfs^{-1}) \neq 0$ for all $f \in F$.
Then $s^{-1}esf \neq 0$ for all $f \in F$.
But $F$ is an ultrafilter and so by Lemma~\ref{lem:exel}
we have that $s^{-1}es \in F$ giving $e \in sFs^{-1}$.

The proof of (2) follows immediately from the proof of (1).
\end{proof}

Our next result establishes that our algebraic notion of support agrees with the topological one.
On a point of notation, if $Y$ is a subset of a topological space then $\mathsf{cl}(Y)$ denotes the {\em closure} of that subset.

\begin{proposition}\label{prop:tempest} Let $S$ be a fundamental Boolean inverse $\wedge$-monoid.
For each $s$, we have that
$$U_{\sigma (s)} = \mathsf{cl}(\{ F \colon F \in U_{s^{-1}s} \mbox{ and } sFs^{-1} \nsubseteq F   \}).$$
\end{proposition}
\begin{proof} Recall that if $F$ is an ultrafilter in a Boolean algebra then either $e \in F$ or $\bar{e} \in F$.
Let $F$ be such that $F \in U_{s^{-1}s}$
and
$sFs^{-1} \nsubseteq F$.
If $\phi (s) \in F$ then 
$s \in F^{\uparrow}$ and so $sFs^{-1}  \subseteq F$, which is a contradiction.
Thus $\overline{\phi (s)} \in F$ which, together with the fact that $s^{-1}s \in F$, gives $\sigma (s) \in F$.

To prove the reverse inclusion, 
put 
$$Y = \{ F \colon  F \in U_{s^{-1}s} \mbox{ and }  sFs^{-1} \nsubseteq F   \}.$$ 
Let $F \in U_{\sigma (s)}$.
We show that every open set containing $F$ intersects $Y$.
It is enough to restrict attention to those open sets $U_{e}$ where $e \in F$.
Suppose that $Y \cap U_{e} = \emptyset$.
Observe that $G \in U_{e}$ contains $s^{-1}s$.
Thus for every $G \in U_{e}$ we have that $sGs^{-1} \subseteq G$.
Hence the set $Z = V_{s} \cap \mathbf{d}^{-1}(V_{e})$ is an open subset of $\mbox{Iso}(\mathsf{G}(S))$.
It follows by Theorem~\ref{thm:one2} that every ultrafilter in $Z$ is idempotent.
Thus $(sF^{\uparrow})^{\uparrow}$ is an idempotent ultrafilter.
We may therefore find $f \in F$ such that $f \leq s$.
But then $f \leq \phi (s) \leq s$.
Hence $\phi (s) \in F$.
But this contradicts the fact that $\sigma (s) \in F$.
It follows that $Y \cap U_{e} \neq \emptyset$.
We have therefore proved that $U_{\sigma (s)}$ is the closure of $Y$.
\end{proof}

The next result is important in translating from topology to algebra.

\begin{lemma}\label{lem:pingo} Let $S$ be a fundamental Boolean inverse $\wedge$-monoid and let $g$ be a unit.
\begin{enumerate}

\item Let $F \subseteq E(S)$ be an ultrafilter such that $gFg^{-1} \neq F$.
Then there exists an idempotent $e \in F$ such that $e \perp geg^{-1}$.

\item  Let $F \subseteq E(S)$ be an ultrafilter such that $\sigma (g) \in F$.
Then for each $f \in F$ where $f \leq \sigma (g)$, 
there exists $0 \neq e \leq f$ such that $e \perp geg^{-1}$.

\end{enumerate}
\end{lemma} 
\begin{proof} (1) Since the ultrafilters $F$ and $gFg^{-1}$ are distinct and the structure space is Hausdorff,
there exist non-zero idempotents $e \perp f$ such that $F \in U_{e}$ and $gFg^{-1} \in U_{f}$.
Since $e \notin gFg^{-1}$, there exists $gig^{-1} \in gFg^{-1}$, where $i \in F$, such that $e (gig^{-1}) = 0$.
Put $j = ie$.
Then $j \in F$ and $j(gjg^{-1}) = 0$.

(2) Let $f \in F$ where $f \leq \sigma (g)$.
By Proposition~\ref{prop:tempest}, 
the open set $U_{f}$ contains an element $G$ such that $gGg^{-1} \neq G$.
Thus by (1), there is an idempotent $e \in G$, which can also be chosen to satisfy $e \leq f$, such that $e(geg^{-1}) = 0$.
\end{proof}


\begin{proposition}\label{prop:two} Let $S$ be a Boolean inverse $\wedge$-monoid.
Then $\mathsf{G}(S)$ is principal if and only if 
$$U_{\phi (s)} = \{ F \colon F \in U_{s^{-1}s} \mbox{ and } sFs^{-1}  \subseteq F   \},$$
for all $s \in S$.
\end{proposition}
\begin{proof} Suppose first that 
$$U_{\phi (s)} = \{ F \colon F \in U_{s^{-1}s} \mbox{ and }  sFs^{-1}  \subseteq F   \},$$
for all $s \in S$.
Let $A,B \in \mathsf{G}(S)$ such that 
$\mathbf{d}(A) = \mathbf{d}(B)$ 
and
$\mathbf{r}(A) = \mathbf{r}(B)$.
Let $\mathbf{d}(A) = \mathbf{d}(B) = F^{\uparrow}$ where $F \subseteq E(S)$ is an ultrafilter.
Then 
$A = (aF)^{\uparrow}$ for any $a \in A$,
and
$B = (bF)^{\uparrow}$ for any $b \in B$.
By assumption, $(aFa^{-1})^{\uparrow} = (bFb^{-1})^{\uparrow}$.
It is easy to check that $a^{-1}bF (a^{-1}b)^{-1} \subseteq F$,
and that $(a^{-1}b)^{-1}a^{-1}b \in F$.
Thus, by assumption, $\phi (a^{-1}b)  \in F$.
It follows that $a^{-1}b \in F^{\uparrow}$.
Hence $aa^{-1}b \in A$ and so $b \in A$.
But then it is immediate that $A = B$.

To prove the converse, assume that $\mathsf{G}(S)$ is principal.
It is enough to prove that 
$$ \{ F \colon F \in U_{s^{-1}s} \mbox{ and } sFs^{-1}  \subseteq F   \}
\subseteq
U_{\phi (s)}.$$ 
Let $sFs^{-1} \subseteq F$ where $s^{-1}s \in F$.
Put $A = (sF)^{\uparrow}$.
Then $\mathbf{d}(A) = F^{\uparrow}$ and $\mathbf{r}(A) = F^{\uparrow}$.
It follows that $A = F^{\uparrow}$.
Thus there is an idempotent $f \in F$ such that $f \leq s$.
But then $f \leq \phi (f)$ and so $\phi (f) \in F$, as required.
\end{proof}

We shall now translate the above result into an internal condition on $S$.
We adapt the idea contained in the first paragraphs of \cite[Section 7]{Kumjian}.

\begin{lemma}\label{lem:mirren} Let $S$ be a Boolean inverse $\wedge$-monoid.
Then $\mathsf{G}(S)$ is principal if and only if for each $s \in S$ we can write
$\sigma (s) = \bigvee_{i=1}^{m} e_{i}$ where $e_{i}se_{i} = 0$ for $1 \leq i \leq e$. 
\end{lemma}
\begin{proof} Suppose first that $\mathsf{G}(S)$ is principal.
Let $s \in S$.
Define 
$$L = \{e \leq s^{-1}s \colon 0 = ese \}.$$
We prove that
$$U_{\sigma (s)} = \bigcup_{e \in L} U_{e}.$$
Let $F \in U_{\sigma (s)}$.
Then $s^{-1}s \in F$ and so $sFs^{-1} \nsubseteq F$ by Proposition~\ref{prop:two}.
It follows that there exist $e,f \in F$ such that
$e (sfs^{-1}) = 0$.
Put $i = ef s^{-1}s $.
Then $i \in F$ and $i \leq s^{-1}s$.
Furthermore, $i (sis^{-1}) = 0$ since $i \leq e$ and $sis^{-1} \leq sfs^{-1}$.
Thus $isi = 0$.
We have shown that $i \in L$ and $F \in U_{i}$.
Thus the lefthand side is contained in the righthand side.
To prove the reverse inclusion, let $F \in U_{e}$ where $e \in L$.
Then $e \in F$ and $e \leq s^{-1}s$ and $ese = 0$.
Suppose that $\phi (s) \in F$.
Then $s \in F^{\uparrow}$ and so $0 = ese \in F^{\uparrow}$, which is a contradiction.
Thus $\overline{\phi (s)} \in F$ and so $\sigma (s) \in F$.
Having proved that 
$$U_{\sigma (s)} = \bigcup_{e \in L} U_{e}$$
we now apply compactness and Lemma~\ref{lem:pele} to deduce that
$$U_{\sigma (s)} = \bigcup_{i=1}^{m} U_{e_{i}} = U_{\bigvee_{i=1}^{m}e_{i}}$$
where $e_{1}, \ldots, e_{n} \in L$. 
Thus
$\sigma (s) = \bigvee_{i=1}^{m} e_{i}$ where $e_{i}se_{i} = 0$ for $1 \leq i \leq e$,
as required.

To prove the converse, we use Proposition~\ref{prop:two}.
Let $F \in U_{s^{-1}s}$ such that $sFs^{-1} \subseteq F$.
Suppose that $\phi (s) \notin F$.
Then $\sigma (s) \in F$.
By assumption, 
$\sigma (s) = \bigvee_{i=1}^{m} e_{i}$ where $e_{i}se_{i} = 0$ for $1 \leq i \leq e$. 
But $F$ is an ultrafilter and so a prime filter.
It follows that $e_{i} \in F$ for some $i$.
Then $se_{i}s^{-1} \in F$ and so $0 = e_{i}se_{i}s^{-1} \in F$, which is a contradiction.
Thus the result follows by Proposition~\ref{prop:two}.
\end{proof}

\begin{remark} {\em Boolean inverse $\wedge$-monoids that satisfy the above condition are the analogues of 
those inverse semigroups that act {\em relatively freely} \cite[Chapter 1, Proposition 2.13]{Renault}.}
\end{remark}

We shall reformulate the above lemma in a more striking form.

\begin{theorem}\label{thm:infinitesimal} Let $S$ be a Boolean inverse $\wedge$-monoid.
Then $\mathsf{G}(S)$ is principal if and only if for each $s \in S$ we have that $s = e \vee s_{1} \vee \ldots \vee s_{m}$
where $e$ is an idempotent, 
each $s_{i}$ is an infinitesimal for $1 \leq i \leq m$,
and $e \perp (s_{1} \vee \ldots \vee s_{m})$.
We may also choose the $s_{i}$ to be pairwise orthogonal.
\end{theorem}
\begin{proof} Suppose first that $\mathsf{G}(S)$ is principal.
Then by Lemma~\ref{lem:mirren}, 
for each $s \in S$ we have that
$\sigma (s) = \bigvee_{i=1}^{m} e_{i}$ where $e_{i}se_{i} = 0$ for $1 \leq i \leq e$. 
By Lemma~\ref{lem:cooper}, we may write
$$s = \phi (s) \vee se_{1} \vee \ldots \vee se_{m}.$$
Now $e_{i}se_{i} = 0$  which gives $(se_{i})^{2} = 0$, and so the $se_{i}$ are infinitesimals.

We now prove the converse.
Assume that each element $s$ can be written $s = e \vee s_{1} \vee \ldots \vee s_{m}$ where $e$ is idempotent and $s_{i}$ infinitesimal. 
We deduce that $\mathsf{G}(S)$ is principal using Lemma~\ref{lem:mirren}. 
We prove first that in fact $e = \phi (s)$.
Let $f \leq s$ be any idempotent.
Then $f = f \wedge s$ and so $f = (f \wedge e) \vee (f \wedge s_{1}) \vee \ldots \vee (f \wedge s_{m})$
by Lemma~\ref{lem:order_properties}.
It follows that each $f \wedge s_{i}$ is an idempotent less than an infinitesimal and so must be 0.
Thus $f \leq e$ and so $e = \phi (s)$
It follows that $s =  \phi (s) \vee (s_{1} \vee \ldots \vee s_{m})$, an orthogonal join.
But $s = \phi (s) \vee s \sigma (s)$, an orthogonal join, by Lemma~\ref{lem:cooper}.
It follows that $s\sigma (s) = s_{1} \vee \ldots \vee s_{m}$.
Thus
$\sigma (s) = s_{1}^{-1}s_{1} \vee \ldots \vee s_{m}^{-1}s_{m} \leq s^{-1}s$
by Lemma~\ref{lem:order_properties}.
Put $e_{i} = s_{i}^{-1}s_{i}$.
Then $s_{i} = se_{i}$.
Hence $se_{i}se_{i} = 0$ and so $e_{i}se_{i} = 0$, as required.
The fact that we may choose the infinitesimals to be pairwise orthogonal follows from Lemma~\ref{lem:bar}
and the observation  that the set of infinitesimals forms an order ideal.
\end{proof}

\begin{remark}{\em The significance of Theorem~\ref{thm:infinitesimal}
may be explained as follows.
In a finite symmetric inverse monoid $I(X)$ each element can be written as a finite orthogonal join
of an idempotent and infinitesimals.
To see why, note that partial bijections of the form $x \mapsto y$, where $x,y \in X$ and $x \neq y$, are infinitesimals,
and that the partial bijections of the form $x \mapsto x$ are idempotents.
This agrees with the above result because the associated groupoid is principal being just $X \times X$.
Thus Boolean inverse $\wedge$-monoids where $\mathsf{G}(S)$ is a principal groupoid 
may be regarded as direct generalizations of finite symmetric inverse monoids.}
\end{remark}

There is no term for Boolean inverse $\wedge$-monoids that satisfy the algebraic condition of Theorem~\ref{thm:infinitesimal} so we introduce one.\\

\noindent
{\bf Definition.} 
A Boolean inverse monoid is said to be {\em basic} if each non-zero element is a finite join of infinitesimals and an idempotent.\\

\begin{lemma}\label{lem:basic} Let $S$ be a basic Boolean inverse monoid.
\begin{enumerate}

\item It is a $\wedge$-monoid.

\item It is piecewise factorizable.

\item It is fundamental.

\end{enumerate}
\end{lemma} 
\begin{proof} (1) Let $s$ be any non-zero element.
By assumption $s = a_{1} \vee \ldots \vee a_{m} \vee e$ where $a_{1}, \ldots, a_{m}$ are infinitesimals and $e$ is an idempotent.
Let $f \leq s$ be any idempotent.
Then $f = f \wedge s$.
It follows by Lemma~\ref{lem:order_properties},
that we have 
$$f = (f \wedge a_{1}) \vee \ldots \vee (f \wedge a_{m}) \vee (f \wedge e).$$
The only idempotent less than an infinitesimal is zero.
Thus $f = f \wedge e$ and so $f \leq e$.
Thus $e$ is the largest idempotent less than or equal to $s$.
We now use Proposition~\ref{prop:leech} to deduce that $S$ is a $\wedge$-monoid.

(2)  From the definition and Lemma~\ref{lem:spooks}.

(3) Let $s$ centralize the idempotents.
We can write $s = e \vee s_{1} \vee \ldots \vee s_{n}$, an orthogonal join, where $e$ is an idempotent and the $s_{i}$ are infinitesimals.
By assumption $\mathbf{d}(s_{i})s = s\mathbf{d}(s_{i})$.
But $s\mathbf{d}(s_{i}) = s_{i}$.
Thus $\mathbf{d}(s_{i})s = s_{i}$.
But $\mathbf{r}(\mathbf{d}(s_{i})s) = \mathbf{r}(s_{i})$
and $\mathbf{r}(\mathbf{d}(s_{i})s) \leq \mathbf{d}(s_{i})$.
Thus $\mathbf{r}(s_{i}) \leq \mathbf{d}(s_{i})$.
By Lemma~\ref{lem:spooks}, however, we have that $\mathbf{r}(s_{i}) \perp \mathbf{d}(s_{i})$.
It follows that $\mathbf{r}(s_{i}) = 0$ and so $s_{i} = 0$.
We have proved that $s = e$, an idempotent, as required.
\end{proof}

We could easily deduce the following by what we proved above, but here we give a direct proof. 

\begin{proposition}\label{prop:basic} Let $S$ be a Boolean inverse monoid.
Then its associated Boolean groupoid $\mathsf{G}(S)$ is principal if and only if $S$ is basic.
\end{proposition}
\begin{proof} Suppose that $S$ is basic.
We prove that the local groups of the groupoid  $\mathsf{G}(S)$ are trivial which is equivalent to $\mathsf{G}(S)$ being principal.
Let $A$ be an ultrafilter in $S$ such that $\mathbf{d}(A) = \mathbf{r}(A)$.
Let $s \in A$.
Since $A$ is a prime filter and using the assumption that $S$ is basic there are two possibilities: either $A$ contains an infinitesimal or $A$ contains an idempotent. 
Suppose that $A$ contains an infinitesimal $a$.
Then $a^{-1}a,aa^{-1} \in \mathbf{d}(A) = \mathbf{r}(A)$ and so $(a^{-1}a)(aa^{-1}) \in \mathbf{d}(A) = \mathbf{r}(A)$.
But this is impossible since this product is zero.
It follows that $A$ contains an idempotent and so is itself idempotent.
Thus the local groups are trivial and the groupoid is principal.

Suppose that $\mathsf{G}(S)$ is principal.
We prove first that every non-idempotent ultrafilter in $S$ contains an infinitesimal.
By principality, such an ultrafilter $A$ must satisfy $\mathbf{d}(A) \neq \mathbf{r}(A)$.
Put $E^{\uparrow} = \mathbf{d}(A)$ and  $F^{\uparrow} = \mathbf{r}(A)$
where $E$ and $F$ are ultrafilters in $E(S)$.
Since the structure space of $E(S)$ is Hausdorff there are open sets $U_{e}$ and $U_{f}$ such that
$E \in U_{e}$ and $F \in U_{f}$ and $e \wedge f = 0$.
Let $a \in A$ be arbitrary.
Then $fae \in A$.
Then $\mathbf{d}(fae) \leq e$ and $\mathbf{r}(fae) \leq f$.
It follows that $fae$ is an infinitesimal.
Now let $s \in S$ be an arbitrary non-zero element.
We show that we can write it as a finite join of an idempotent and infinitesimals.
Cleraly, $V_{s} = U \cup V$ where $U$ is the set of those ultrafilters $A$ such that $\mathbf{d}(A) \neq \mathbf{r}(A)$
and $V$ is the set of those ultrafilters where $\mathbf{d}(A) = \mathbf{r}(A)$.
By construction $U \cap V = \emptyset$.
Consider first the set $V$.
If it is non-empty, then by principality it contains only idempotent ultrafilters.
It follows that $V = V_{s} \cap \mathsf{G}(S)_{o}$ and so is an open set and, consequently, $U$ is an open set.
As an open set, $V$ can be written as a union of basic open sets $V_{a_{i}}$ where $V_{a_{i}}$ contains only idempotent ultrafilters.
It follows by Lemma~\ref{lem:pele}, that $a_{i}$ must be an idempotent.
Thus if $V$ is non-empty, we may write $V = \bigcup_{i \in I} V_{e_{i}}$ for some idempotents $e_{i}$, where in particular $e_{i} \leq s$.
We now turn to the set $U$.
If it is non-empty, then by our result above, every ultrafilter in it must contain an infinitesimal.
Now infinitesimals form an order ideal thus we may write $V = \bigcup_{j \in J} V_{a_{j}}$ where the $a_{j}$ are infinitesimals,  where in particular  $a_{i } \leq s$.
It follows now by compactness that $s$ can be written as a finite join of infinitesimals and an idempotent and so $S$ is basic as claimed.
\end{proof}



\section{Proof of Theorem~\ref{thm:TWO}}

The main goal of this section is to describe the structure of the group of units of a fundamental Tarski inverse monoid and relate it to the structure of the monoid itself.

\subsection{Fundamental Boolean inverse $\wedge$-monoids}

Denote by $\mbox{Aut}(E)$ the group of automorphisms of the Boolean algebra $E$.
The following was proved by Stone himself. 
See \cite{GH}.

\begin{lemma}\label{lem:homeo} Let $E$ be a Boolean algebra with associated Stone space $\mathsf{X}(E)$.
Then the groups $\mbox{\rm Aut}(E)$ and $\mbox{\rm Homeo}(\mathsf{X}(E))$ are isomorphic.
\end{lemma}

We now generalize this result.
Given a Boolean algebra $E$, 
we may form the Munn semigroup $T_{E}$, a Boolean inverse monoid by Proposition~\ref{prop:munn-one}.
Associated with $E$ is its Stone space $\mathsf{X}(E)$.
Define $\mathsf{I}(\mathsf{X}(E))$ to be the inverse monoid of all partial homeomorphisms between the clopen subsets of $\mathsf{X}(E)$.
This is a Boolean inverse monoid.

\begin{proposition}\label{prop:munn-again} 
For each Boolean algebra $E$, the inverse monoids  $T_{E}$ and  $\mathsf{I}(\mathsf{X}(E))$ are isomorphic.
\end{proposition}
\begin{proof} It is a classical result that every principal ideal $e^{\downarrow}$ in $E$ has corresponding Stone space $U_{e}$.
Essentially, ultrafilters in $e^{\downarrow}$ may be enlarged to ultrafilters in $E$ that contain $e$, and conversely.
Specifically, if $F$ is an ultrafilter in $e^{\downarrow}$ then $F^{\uparrow}$ is an ultrafilter in $E$ containing $e$.
Conversely, if $A$ is an ultrafilter in $E$ containing $e$ then $A \cap e^{\downarrow}$ is an ultrafilter in $e^{\downarrow}$.
An order isomorphism between $e^{\downarrow}$ and $f^{\downarrow}$ leads to a homeomorphism between $U_{e}$ and $U_{f}$, and conversely.
Thus there is a bijection between $T_{E}$ and  $\mathsf{I}(\mathsf{X}(E))$.
This bijection preserves the groupoid product, the natural partial order, and the Boolean algebra structures.
It follows that it induces an isomorphism.
\end{proof}

\begin{remark} {\em It is probably worth being explicit about this:
the group of units of $T_{E}$ is the group of all automorphisms of the Boolean algebra $E$.}
\end{remark}

Proposition~\ref{prop:munn-again} makes the proof of the following straighfoward using properties of continuous functions.

\begin{lemma}\label{lem:fullness} Let $E$ be a Boolean algebra.
Let $e_{1}, \ldots, e_{m}$ be pairwise, non-zero, orthogonal idempotents of $E$ with join the identity.
Let $g_{1}, \ldots, g_{m} \in \mbox{\rm Aut}(E)$.
Suppose that $g_{1}e_{1}g_{1}^{-1}, \ldots, g_{m}e_{m}g_{m}^{-1}$ are also orthogonal with join the identity.
Then $g = \bigvee_{i=1}^{m} g_{i}e_{i}$ is also an element of $\mbox{\rm Aut}(E)$ 
\end{lemma}

In the light of Lemma~\ref{lem:fullness}, the next definition makes sense.
A subgroup $G \leq \mbox{Aut}(E)$ is said to be {\em full} if the following condition holds.
Let $e_{1}, \ldots, e_{m}$ be pairwise, non-zero, orthogonal idempotents of $E$ with join the identity.
Let $g_{1}, \ldots, g_{m} \in G$.
Suppose that $g_{1}e_{1}g_{1}^{-1}, \ldots, g_{m}e_{m}g_{m}^{-1}$ are also orthogonal with join the identity.
Then $g = \bigvee_{i=1}^{m} g_{i}e_{i} \in G$.

Let $G$ be a group and $E$ a Boolean algebra.
An {\em action} of $G$ on $E$ is determined by a homomorphism from $G$ to the group of automorphisms of $E$.
We denote the action by $e \mapsto g \cdot e$.
The action is {\em faithful} if $g$ acts as the identity on $E$ if and only if $g$ is the identity.
Let $S$ be a Boolean inverse monoid.
Then $\mathsf{U}(S)$, the group of units of $S$, acts on the semilattice of idempotents $E(S)$ by conjugation $e \mapsto geg^{-1}$.
We call $(\mathsf{U}(S),E(S))$ the {\em natural action associated with $S$}.

\begin{lemma}\label{lem:nice} Let $S$ be a Boolean inverse monoid.
\begin{enumerate}
\item  Suppose that $g$ is a unit such that $gfg^{-1} = f$ for all $f \leq e$.
Then $ge \in Z(E(eSe))$.
\item If $S$ is fundamental then the natural action of $\mathsf{U}(S)$ on $E(S)$ is faithful.
\end{enumerate}
\end{lemma}
\begin{proof} {1} By assumption, $eg = ge$ and so $eg = ege \in eSe$.
Let $f \leq e$.
Then $(ege)f = egfe = efge = f(ege)$.
Thus $ge \in Z(E(eSe))$.
(2) This is immediate by part (1) above.
\end{proof}

Let an action of a group $G$ on a Boolean algebra $E$ be given.
Then a function $\phi \colon G \rightarrow E$ is called an {\em operator (with respect to this action)} if it satisfies the following axioms.
\begin{description}
\item[{\rm (O1)}]  $\phi (1) = 1$.
\item[{\rm (O2)}]  $\phi (g^{-1}) = \phi (g)$.
\item[{\rm (O3)}] $\phi (g) \phi (h) \leq \phi (gh)$.
\item[{\rm (O4)}] For each $g \in G$ and $e \in E$ we have that $\phi (g)e \leq g \cdot e$.
\item[{\rm (O5)}] The element $g$ fixes the principal order ideal $e^{\downarrow}$ pointwise if and only if $e \leq \phi (g)$.
\end{description}
A triple $(G,E,\phi)$ consisting of a full subgroup $G$ of $\mbox{Aut}(E)$ that acts faithfully on the Boolean algebra $E$
together with an operator $\phi \colon G \rightarrow E$ is called an {\em armature}\footnote{Used in the sense of sculpture where it means a framework around which the sculpture is constructed.}.
If $E$ is also a Tarski algebra and the group $G$ is countable, we say that it is a {\em Tarski armature}.

\begin{proposition}\label{prop:ashleigh} Let $S$ be a fundamental Boolean inverse $\wedge$-monoid with fixed point operator $\phi$.
Let $\phi \colon \mathsf{U}(S) \rightarrow E(S)$ also denote the restriction of the fixed-point operator to the group of units of $S$.
Then $\mathsf{A}(S) = (\mathsf{U}(S),E(S),\phi)$ is an armature.
\end{proposition}
\begin{proof} By Lemma~\ref{lem:nice}, the natural action $(\mathsf{U}(S),E(S))$ is faithful.
Thus $\mathsf{U}(S)$ is isomorphic under the natural action to a subgroup of $\mbox{Aut}(E(S))$.
This is a full subgroup because $S$ has all finite non-empty compatible joins.
The proofs of (O1), (O2) and (O3) are straightfoward.
The proof of (O4) follows from the fact that $\phi (g) \leq g,g^{-1}$.
We prove (O5).
Suppose first that $e \leq \phi (g)$.
Then from the definition of $\phi (g)$ we have that $e = eg = ge$.
Thus $ge = eg$ and so $geg^{-1} = e$.
Conversely, suppose that $g$ fixes the principal order ideal $e^{\downarrow}$ pointwise. 
Then by Lemma~\ref{lem:nice}, we have that $eg$ is in the centralizer of $E(eSe)$.
Thus $eg = ege = ge$ is an idempotent by Lemma~\ref{lem:local_monoid}.
Hence $ege = eg \leq e$.
It follows that $e \leq eg^{-1} \leq g^{-1}$
from which we get that $e \leq \phi (g^{-1}) = \phi (g)$.
\end{proof}

We now prove that armatures are precisely what is needed to construct piecewise factorizable fundamental Boolean inverse $\wedge$-monoids.
The following lemma is useful.

\begin{lemma}\label{lem:wattar} Let $S$ and $T$ be two wide factorizable inverse submonoids of the inverse monoid $V$.
Suppose that $\mathsf{U}(S) = \mathsf{U}(T)$.
Then $S = T$.
 \end{lemma}
\begin{proof} We prove that $S \subseteq T$.
The reverse inclusion follows by symmetry.
Let $s \in S$.
Then since $S$ is factorizable $s = ge$ for some $g \in \mathsf{U}(S)$ and $e \in E(S)$.
Since  
$\mathsf{U}(S) = \mathsf{U}(T)$
and
$E(S) = E(S)$
we have that $S = ge \in T$, as required.
\end{proof}

\begin{theorem}\label{thm:armatures-pf} There is a bijective correspondence between 
(isomorphism classes of) piecewise factorizable fundamental Boolean inverse $\wedge$-monoids
and 
(isomorphism classes of) armatures.
\end{theorem}
\begin{proof} In Proposition~\ref{prop:ashleigh}, we constructed an armature from a  fundamental Boolean inverse $\wedge$-monoid.
We now show how to go in the opposite direction.
Let $(G,E,\phi)$ be an armature.
We shall construct a piecewise factorizable fundamental Boolean inverse $\wedge$-monoid $S = S(G,E,\phi)$ such that $\mathsf{A}(S) = (G,E,\phi)$ in three steps.
Since the action of $G$ on $E$ is faithful, we may assume that $G$ is a subgroup of $\mathsf{U}(T_{E})$.
We shall construct $S$ as a wide inverse submonoid of $T_{E}$ so that it will automatically be fundamental by Theorem~\ref{them:classic}.

(Step~1).
Define $T = G^{\downarrow}$.
By Lemma~\ref{lem:factorizable}, 
this is a factorizable wide inverse submonoid of $T_{E}$ with group of units $G$ and semilattice of idempotents isomorphic to $E$.
By Lemma~\ref{lem:wattar}, $T$ is determined by $G$.
Elements of $T$ are products inside $T_{E}$ of the form $g\alpha$ where $g \in G$ is an automorphism of $E$
and $\alpha$ is an identity map on a principal order ideal $e^{\downarrow}$ of $E$.
A simple calculation shows that $g\alpha g^{-1}$ is the identity map on $g(e)^{\downarrow}$.
Thus conjugation of idempotents within the inverse monoid $T$ is essentially the same as the action of
elements of $G$ on $E$.
Thus we may assume that the action is given by conjugation.

(Step~2).
We prove that $T$ is a $\wedge$-monoid by constructing a fixed-point operator.
Let $a \in T$.
Then $a = ge$ where $g \in G$ and $e \in E$.
Define $\phi (a) = \phi (g)e$ where $\phi (g)$ is defined via the armature.
Observe that if $ge = hf$ in $T$ then by taking domains $e = f$.
Suppose that $ge = he$.
We prove that $\phi (g)e = \phi (h)e$.
Let $f \leq e$.
Then $gfg^{-1} = gefeg^{-1} = hefeh^{-1} = hfh^{-1}$.
Thus $h^{-1}g$ fixes the principal order ideal $e^{\downarrow}$ pointwise (under conjugation)
and by axiom (O5) this implies that $e \leq \phi (h^{-1}g)$.
By axiom (O3), we have that $\phi (h)e \leq \phi (h) \phi (h^{-1}g) \leq \phi (g)$. 
It follows that $\phi (h)e \leq \phi (g)e$.
By symmetry $\phi (h)e = \phi (g)e$.
Define $\phi (ge) = \phi (g)e$.
Thus $\phi$ is a well-defined function from $S$ to $E(S)$.
We prove that it is a fixed-point operator by checking the axioms in turn.
(FPO1) holds.
We prove that $\phi (g) \leq g$ in the inverse monoid $T$.
By axiom (O4), we have that $\phi (g) \leq g \phi (g)g^{-1}$ and $\phi (g^{-1}) \leq g^{-1} \phi (g^{-1})g$.
Thus $\phi (g)g \leq g \phi (g)$ and by axiom (O2) we also have that $g \phi (g) \leq \phi (g)g$.
We deduce that $\phi (g) g = g \phi (g)$.
If we prove that $\phi (g)g$ is an idempotent then we shall have proved that $\phi (g) = g \phi (g)$ and so that $\phi (g) \leq g$.
Let $e \leq \phi (g)$.
Then by axiom (O5), we have that $geg^{-1} = e$.
By Lemma~\ref{lem:nice}, this implies that $\phi (g)g$ is in the centralizer of $E(\phi (g)T \phi (g))$.
But $T$ is fundamental and so $\phi (g)g$ is an idempotent by  Lemma~\ref{lem:local_monoid}.
(FPO2) holds.
Let $a \in T$ and $e \in E(T)$ such that $e \leq a$.
Assume that $a = gf$.
In particular, $e \leq f$.
Thus $e = ge$.
It follows that for all $i \leq e$ we have that $i = gi$.
Thus the principal order ideal $e^{\downarrow}$ is fixed pointwise under conjugation by $g$.
By axiom (O4), we have that $e \leq \phi (g)$.
Hence $e \leq \phi (g)f = \phi (a)$.
We have therefore proved that $\phi$ is a fixed-point operator and so by Proposition~\ref{prop:leech}
it follows that $T$ is an inverse $\wedge$-monoid.

(Step~3).
Define $S = (G^{\downarrow})^{\vee}$.
Then $S$ is an  inverse submonoid of $T_{E}$
and the fact that $S$ is a distributive inverse monoid follows from  
Proposition~\ref{prop:munn-one} and Proposition~\ref{prop:building}.
It is a wide inverse submonoid of a Munn inverse monoid and so it is Boolean and fundamental.
We proved above that $G^{\downarrow}$ is an inverse $\wedge$-monoid.
Thus by Corollary~\ref{cor:QI}, we have that $S$ is an inverse $\wedge$-monoid
which is piecewise factorizable by construction.
The armature associated with $S$ is $(G,E,\phi)$ upto isomorphism. 

Let $S$ be a piecewise factorizable fundamental Boolean inverse $\wedge$-monoid.
Without loss of generality, we may assume by Theorem~\ref{them:classic} that $S$ is given as a wide
inverse submonoid of $T_{E(T)}$.
Clearly $S = (\mathsf{U}(S)^{\downarrow})^{\vee}$.
It is now clear by our construction above that $S = S(\mathsf{U}(S),E(S),\phi)$.
\end{proof}

We shall now connect armatures with certain groups of homeomorphisms.
Let $(G,E,\phi)$ be an armature.
We denote the Stone space of $E$ by $\mathsf{X}(E)$.
We may define an action of $G$ on $\mathsf{X}(E)$
by $g \cdot F = \{g \cdot f \colon f \in F\}$
where $F\subseteq E$ is an ultrafilter.
Observe that $g \cdot U_{e} = U_{g \cdot e}$.
Thus the action is by means of homeomorphisms.
We have therefore constructed an action $G \times \mathsf{X}(E) \rightarrow \mathsf{X}(E)$.

\begin{lemma}\label{lem:rain} Let $(G,E,\phi)$ be an armature.
Let $g \in G$.
Then $g$ fixes $e^{\downarrow}$ pointwise if and only if $g$ fixes $U_{e}$ pointwise.
In particular, the action of $G$ on $\mathsf{X}(E)$ is faithful.
\end{lemma}
\begin{proof} Suppose first that $g$ fixes $e^{\downarrow}$ pointwise.
Then by axiom (O5), we have that $e \leq \phi (g)$.
Let $F \in U_{e}$.
Then $\phi (g) \in F$.
Let $f \in F$.
Then by (O4), $\phi (g) f \leq g \cdot f$ and so $g \cdot f \in F$.
It follows that $g \cdot F = F$.
Thus $g$ fixes $U_{e}$ pointwise.
Suppose now that $g$ fixes $U_{e}$ pointwise.
Let $f \leq e$.
Then by assumption, $g \cdot U_{f} = U_{f}$ but this equals $U_{g \cdot f}$.
It follows that $g \cdot f = f$ and so $g$ fixes $e^{\downarrow}$ pointwise.
Now suppose that $g$ fixes every ultrafilter.
Then $g$ fixes $U_{1}$ pointwise and so by the above result $g$ fixes every idempotent pointwise
which, by the assumption of faithfulness, shows that $g = 1$.
\end{proof}

\begin{lemma}\label{lem:hagel} Let $(G,E,\phi)$ be an armature.
Then
$$U_{\overline{\phi (g)}} = \mathsf{cl}(\{ F \colon F \in \mathsf{X}(E) \mbox{ and } g \cdot F \neq F   \}).$$
\end{lemma}
\begin{proof}
Put $Y = \{ F \in \mathsf{X}(E) \colon g \cdot F \neq F \}$.
Let $F \in Y$ and suppose that $\phi (g) \in F$.
Let $f \in F$.
Then by axiom (O5), we have that $\phi (g)f \leq g \cdot f$.
It follows that $g \cdot F = F$ which is a contradiction.
Since $F$ is an ultrafilter, it follows that $\overline{\phi (g)} \in F$.
Thus $Y \subseteq U_{\overline{\phi (g)}}$.
Let $F \in U_{\overline{\phi (g)}}$ and suppose that $F \in U_{f}$ where $f$ is any idempotent.
If $g$ fixed $U_{f}$ pointwise, 
then by Lemma~\ref{lem:rain} we would have $f \leq \phi (g)$ and so $\phi (g) \in F$ which is a contradiction.
 It follows that $U_{f} \cap Y \neq \emptyset$.
\end{proof}

The set 
$\mathsf{cl}(\{ F \colon F \in \mathsf{X}(E) \mbox{ and } g \cdot F \neq F   \})$
is nothing other than $\mbox{supp}(g)$.

\begin{theorem}\label{thm:hail} 
There is a bijective correspondence between (isomorphism classes of) armatures  $(G,E,\phi)$ and Stone subgroups of $\mbox{\rm Homeo}(\mathsf{X}(E))$.
\end{theorem}
\begin{proof} Let $G$ be a subgroup of $\mbox{Homeo}(\mathsf{X}(E))$ in which the support of each element of $G$ is clopen.
By assumption, the support operator $\mbox{\rm supp}$ maps elements of $G$ to elements of the Boolean algebra $E$.
Define $\phi (g) = \overline{\mbox{\rm supp} (g)}$.
We prove that  $(G,E,\phi)$ is an armature.
Axioms (O1) and (O2) are immediate from the definitions.
Axiom (O3) follows from the fact that $\mbox{supp}(gh) \subseteq \mbox{supp}(g) \cup \mbox{supp}(h)$.
Axiom (O4) follows from the fact that $g(e) \subseteq \mbox{supp}(g) \cup e$.
We show that Axiom (O5) holds.
To do this, we shall move between the Boolean algebra $E$ and its Stone space $\mathsf{X}(E)$.
Thus the elements of $E$ are the clopen subsets $\mathsf{X}(E)$.
Suppose that for all $f \subseteq e$ we have that $g(f) = f$.
I shall prove that $\mbox{supp}(g) \subseteq \overline{e}$.
Let $x$ be a point such that $g(x) \neq x$.
Stone spaces are compact Hausdorff spaces that have a basis of clopen sets.
We may therefore find an element $f \in E$ such that $x \in f$ and $f \cap g(f) = \emptyset$.
Suppose that $x \notin \overline{e}$.
Then $x \in e$.
It follows that $x \in e \cap f$.
In particular, $e \cap f \neq 0$.
By assumption, we have that $g(e \cap f) = e \cap f$.
But this contradicts  $f \cap g(f) = \emptyset$.
It follows that $x \notin \overline{e}$.
The result now follows.
We have therefore constructed an armature.
The other direction follows by Lemma~\ref{lem:hagel}.
\end{proof}

If we combine Theorem~\ref{thm:armatures-pf} and Theorem~\ref{thm:hail} and recall that the groups considered should be countable, 
we obtain the following.

\begin{theorem}\label{thm:west} 
There is a bijective correspondence between (isomorphism classes of) piecewise factorizable fundamental Tarski inverse monoids
and Cantor groups.
\end{theorem}

Our notion of an armature may be viewed as an abstraction and generalization of Krieger's {\em unit systems} \cite{Krieger}.
We briefly touch on a special case of our theory that enables us to connect the work of our paper with Krieger's.
See also \cite[Proposition 1.15]{Renault} and \cite{LS}.
Recall that Krieger \cite{Krieger} defines an {\em ample group} to be, amongst other things and in our terminology, a Stone group 
in which the fixed-point set of each element is clopen rather than the closure of the fixed-point set.

\begin{proposition}\label{prop:krieger} Let $S$ be a piecewise factorizable, fundamental Boolean inverse $\wedge$-monoid.
Then the following are equivalent.
\begin{enumerate}
\item $S$ is basic.
\item In the natural action of $\mathsf{U}(S)$ on $\mathsf{X}(E(S))$ the fixed-point set of each element is clopen.
\end{enumerate}
\end{proposition}
\begin{proof} Suppose first that $S$ is basic.
Then for each $s \in S$, we have that
$$U_{\phi (s)} = \{F \colon F \in U_{s^{-1}s} \mbox{ and } sFs^{-1} \subseteq F\}$$
by Proposition~\ref{prop:two}, 
If we restrict to $g \in \mathsf{U}(S)$, we have that
$$U_{\phi (g)} = \{F \colon F \in U_{1} \mbox{ and } gFg^{-1} = F\}.$$
This says that the fixed-point set of each element in the natural action is clopen.
We prove the converse.
Suppose that each element of $\mathsf{U}(S)$ has the property that its fixed-point set under the natural action is clopen.
Thus we are assuming that the set
$$\{F \in \mathsf{X}(S) \colon gFg^{-1} = F \}$$
is clopen for each $g \in \mathsf{U}(S)$.
It follows that the set
$$\{F \in \mathsf{X}(S) \colon gFg^{-1} \neq F \}$$
is also clopen and their union is $\mathsf{X}(S)$.
But by Lemma~\ref{lem:hagel}, we have that 
$$\{F \in \mathsf{X}(S) \colon gFg^{-1} \neq F \} = U_{\overline{\phi (g)}}.$$
Thus 
$$\{F \in \mathsf{X}(S) \colon gFg^{-1} = F \} = U_{\phi (g)}$$
for each element $g \in \mathsf{U}(S)$.

Let $s \in S$.
Now
$$U_{\phi (s)} \subseteq \{ F \colon F \in U_{s^{-1}s} \mbox{ and } sFs^{-1} \subseteq F\}$$
always holds.
We prove the reverse inclusion first in a special case.
Suppose that $s = ge$ where $g$ is a unit and $e = s^{-1}s$.
Let $F \in U_{s^{-1}s}$ be such that $sFs^{-1} \subseteq F$.
Let $f \in F$.
Then $sfs^{-1} \in F$ and so $gfg^{-1} \in F$.
It follows that $gFg^{-1} = F$.
By what we proved above, we have that $\phi (g) \in F$.
But by Lemma~\ref{lem:fpo}, $\phi (s) = \phi (g)e$ and so $\phi (s) \in F$.
Thus the reverse inclusion holds in this special case.
Let $s$ now be an aribitrary element of $S$.
Since $S$ is piecewise factorizable, we can write $s = \bigvee_{i=1}^{m} s_{i}$ where
each $s_{i} = g_{i}e_{i}$ for some unit $g_{i}$ and $e_{i} = \mathbf{d}(s_{i})$.
By Lemma~\ref{lem:fpo}, we have that $\phi (s) =  \bigvee_{i=1}^{m} \phi (s_{i})$.
By our result above
$$U_{\phi (s_{i})} = \{ F \colon F \in U_{s_{i}^{-1}s_{i}} \mbox{ and } s_{i}Fs_{i}^{-1} \subseteq F\}.$$
But by Lemma~\ref{lem:pele}, we have that
$$U_{\phi (s)} = \bigcup_{i=1}^{m} U_{\phi (s_{i})}.$$ 
Thus
$$U_{\phi (s)} = \{ F \colon F \in U_{s^{-1}s} \mbox{ and } sFs^{-1} \subseteq F\}$$
and so $S$ is basic by Proposition~\ref{prop:two}.
\end{proof}

\subsection{$0$-simplifying fundamental Tarski inverse monoids}

We shall prove that the Tarski inverse monoids studied in this section are automatically piecewise factorizable.

\begin{proposition}\label{prop:welsh} Let $S$ be a Boolean inverse $\wedge$-monoid.
Then each ultrafilter of $S$ contains a unit if and only if $S$ is piecewise factorizable.
\end{proposition}
\begin{proof} Suppose first that $S$ is piecewise factorizable.
Let $A$ be any ultrafilter and choose any $s \in A$.
By assumption we may write 
$s = \bigvee_{i=1}^{m} s_{i}$ where for each $s_{i}$ there is a unit $g_{i}$ such that $s_{i} \leq g_{i}$.
But every ultrafilter is prime by Lemma~\ref{lem:molina} and so $s_{i} \in A$ for some $i$.
Consequently $g_{i} \in A$ and so each ultrafilter contains a unit.
To prove the converse, we assume that every ultrafilter contains a unit.
Let $s \in S$ be any non-zero element.
We shall write $V_{s}$ as a union of clopen sets.
Let $A \in V_{s}$.
Then there is some unit $g \in A$.
Thus $g \wedge s \in A$.
We may therefore write
$V_{s} = \bigcup V_{s_{i}}$ where the $s_{i}$ are those elements belonging to the elements of $V_{s}$ which are  beneath units.
By compactness, we may write $V_{s} = \bigcup_{i=1}^{m} V_{s_{i}}$ for some finite union, and the result now follows by Lemma~\ref{lem:pele}.
\end{proof}

Our next result is fundamental since it enables us to construct infinitesimals with specific properties.

\begin{lemma}\label{lem:george} Let $S$ be a $0$-simplifying Tarski inverse monoid.
Let $F \subseteq E(S)$ be an ultrafilter and let $e \in F$.
Then there exists an element $a \in S$ such that
\begin{enumerate}

\item $a$ is an infinitesimal.

\item $a \in eSe$.

\item $a^{-1}a \in F$.

\end{enumerate}
\end{lemma}
\begin{proof}
The idempotent $e$ is non-zero and so, 
since we are working in a Tarski algebra, it cannot be an atom.
Thus there exists $0 \neq f < e$.
The idempotents form a Boolean algebra, and so there exists an idempotent $f'$ such that $e = f \vee f'$ and $f \wedge f' = 0$.
Since $f \vee f' = e \in F$, and $F$ is an ultrafilter and so prime, 
we know that either $f \in F$ or $f' \in F$.
Without loss of generality, we may assume that $f \in F$.
Now $S$ is $0$-simplifying and so $f' \equiv f$ by Lemma~\ref{lem:toby}.
In particular, $f \preceq f'$.
We may therefore find elements $x_{1}, \ldots, x_{m}$ 
such that
$f = \bigvee_{i=1}^{m} \mathbf{d}(x_{i})$ 
and $\mathbf{r}(x_{i}) \leq f'$.
We use the fact that $F$ is a prime filter, to deduce that $\mathbf{d}(x_{i}) \in F$ for some $i$.
Put $a = x_{i}$.
Then $a^{-1}a \leq f$ and $aa^{-1} \leq f'$.
Hence $a^{-1}a \perp aa^{-1}$.
It follows that $a$ is an infinitesimal by Lemma~\ref{lem:spooks}.
Clearly,  $a^{-1}a,aa^{-1} \leq e$ and so $a \in eSe$.
In addition, $a^{-1}a \in F$ by construction.
\end{proof}

The above lemma tells us that infinitesimals are plentiful in $0$-simplifying Tarski inverse monoids
and so by Lemma~\ref{lem:spooks} involutions are plentiful.

\begin{lemma}\label{lem:horsa} Let $S$ be a $0$-simplifying Tarski inverse monoid. 
Let $e$ be any non-zero idempotent.
Then there exist infinitesimals $a,b \in eSe$ such that $ab$ is a restricted product and an infinitesimal.
\end{lemma}
\begin{proof} Every non-zero idempotent is an element of some ultrafilter in $E(S)$.
Thus by Lemma~\ref{lem:george}, we may find an infinitesimal $x \in eSe$.
Similarly, we may find an infinitesimal $b \in \mathbf{d}(x)S\mathbf{d}(x)$.
Put $a = x \mathbf{r}(b)$.
The set of infinitesimals forms an order ideal, and so $a$ is an infinitesimal.
By construction, $ab$ is a restricted product, and since $\mathbf{r}(a) \perp \mathbf{d}(b)$ it is an infinitesimal. 
\end{proof}

\begin{lemma}\label{lem:hengist} In a $0$-simplifying Tarski inverse monoid,  
every non-idempotent ultrafilter contains an infinitesimal or the product of two infinitesimals.
\end{lemma}
\begin{proof} There are two cases to consider.
Suppose first that  $A$ is an ultrafilter such that $A^{-1} \cdot A \neq A \cdot A^{-1}$.
Both  $A^{-1} \cdot A$ and $A \cdot A^{-1}$ are idempotent ultrafilters and distinct by assumption.
Since the groupoid $\mathsf{G}(S)$ is Hausdorff there are compact-open sets $V_{s}$ and $V_{t}$ such that
$A^{-1} \cdot A \in V_{s}$ and $A \cdot A^{-1} \in V_{t}$ where $s \wedge t = 0$.
Since  $A^{-1} \cdot A$ and $A \cdot A^{-1}$ are idempotent ultrafilters, 
we may find idempotents $e$ and $f$ such that $A^{-1} \cdot A \in V_{e}$ and $A \cdot A^{-1} \in V_{f}$  and $e \wedge f = 0$.
Let $a \in A$ and put $b = fae$.
Then $b \in A$ and $b^{2} = 0$.

Now suppose that $A$ is a non-idempotent ultrafilter such that $A^{-1} \cdot A = A \cdot A^{-1} = F^{\uparrow}$, where $F \subseteq E(S)$ is an ultrafilter.
Let $e \in F$.
By Lemma~\ref{lem:george}, there is an infinitesimal $a$ such that $a \in eSe$ and $a^{-1}a \in F$.
Put $B = (aF)^{\uparrow}$, a well-defined ultrafilter containing an infinitesimal where $\mathbf{d}(B) = F^{\uparrow}$.
Put $G = E(\mathbf{r}(B))$.
Then $aa^{-1} \in G$ but $aa^{-1}a^{-1}a = 0$.
It follows that $G \neq F$.
Clearly, $A = (AB^{-1})B$.
But $B$ contains an infinitesimal by construction and $AB^{-1}$ contains an infinitesimal by our first result above.
Thus $A$ contains a product of two infinitesimals.
\end{proof}

Since idempotent ultrafilters contain the identity,
it follows by Lemma~\ref{lem:hengist} and Lemma~\ref{lem:spooks}, 
that every ultrafilter in a $0$-simplifying Tarski inverse monoid contains a unit.
Thus the following is a direct application of Proposition~\ref{prop:welsh}.

\begin{theorem}\label{thm:dory}  
Every $0$-simplifying Tarski inverse monoid is piecewise factorizable.
\end{theorem}

The following result shows the very close connection between the structure of a fundamental Tarski inverse monoid  and the structure of its group of units.
Recall that an action of a discrete group $G$ on a topological space $X$ is said to be {\em minimal} if the closure of every orbit is $X$.

\begin{theorem}\label{thm:regen} Let $S$ be a fundamental Tarski inverse monoid.
Then $S$ is $0$-simplifying if and only if the action of $\mathsf{U}(S)$ on $\mathsf{X}(S) = \mathsf{X}(E(S))$ is minimal.
\end{theorem}
\begin{proof}
Suppose that the Tarski inverse monoid $S$ is $0$-simplifying.
Let $U_{e}$, where $e \neq 0$, be any of the basic open sets of the structure space $\mathsf{X}(S)$.
Let $F$ be any point of $\mathsf{X}(S)$.
We prove that $gFg^{-1} \in U_{e}$ for some $g \in U(S)$.
Let $f \in F$ be any element.
By Lemma~\ref{lem:toby} $e \equiv f$ since $S$ is $0$-simplifying.
Thus in particular $f \preceq e$.
There is therefore a pencil $\{x_{i}\}$ from $f$ to $e$.
Since $F$ is a prime filter, there is an $x_{i} = a$ such that $\mathbf{d}(a) \in F$  and $\mathbf{d}(a) \leq f$ and $\mathbf{r}(a) \leq e$.
But $S$ is piecewise factorizable by Theorem~\ref{thm:dory},
and so we may write
$a = \bigvee_{j=1}^{m} g_{j}e_{j}$ where the $g_{j}$ are units and the $e_{j}$ idempotents.
Now $\mathbf{d}(a) =  \bigvee_{j=1}^{m} \mathbf{d}(g_{j}e_{j})$.
It follows that $\mathbf{d}(g_{j}e_{j}) \in F$ for some $j$
where $\mathbf{d}(g_{j}e_{j}) \leq \mathbf{d}(a) \leq f$
and $\mathbf{r}(g_{j}e_{j}) \leq \mathbf{r}(a) \leq e$.
Put $g = g_{j}$ and $i = e_{j}$.
Now $i \in F$ and $gig^{-1} \leq e$.
Thus $gFg^{-1}$ contains $e$, as required.

To prove the converse, suppose that the associated action $(\mathsf{U}(S),\mathsf{X}(S))$ is minimal.
We prove that $S$ is $0$-simplifying.
Let $e$ and $f$ be non-zero idempotents of $S$.
We shall prove that $e \preceq f$ from which the result will follow by symmetry.
Let $F \in U_{e}$.
Then there exists $g \in \mathsf{U}(S)$ such that $gFg^{-1} \in U_{f}$ by minimality of the group action.
Thus there exists $i \in F$ such that $gig^{-1} = f$.
Put $j = ie$.
Then $j \in F$ and $gjg^{-1} \leq f$.
We therefore have that
$F \in U_{j} \subseteq U_{e}$ and $U_{gjg^{-1}} \subseteq U_{f}$.
We use the fact that $U_{e}$ is compact combined with Lemma~\ref{lem:pele}
to deduce that there is a finite set of units $\{g_{1}, \ldots, g_{m}\}$
and a finite set of idempotents $\{e_{1}, \ldots, e_{m}\}$ such that
$e = \bigvee_{i=1}^{m} e_{i}$ and $g_{i}e_{i}g_{i}^{-1} \leq f$.
It follows that $\{g_{1}e_{1}, \ldots, g_{m}e_{m} \}$ is a pencil from $e$ to $f$.
\end{proof}

If we combine Theorem~\ref{thm:west} with Theorem~\ref{thm:dory} and Theorem~\ref{thm:dory}, we obtain the following.

\begin{theorem}\label{thm:france}
There is a bijective correspondence between (isomorphism classes of) $0$-simplifying fundamental Tarski inverse monoids
and 
(isomorphism classes of) minimal Cantor groups.
\end{theorem}



\begin{thebibliography}{99}


\bibitem{Birget} J.-C.~Birget, The groups of Richard Thompson and  complexity, {\em IJAC} {\bf 14} (2004), 569--626.


\bibitem{BCFS} J. Brown, L.. O. Clark, C. Farthing, A. Sims, Simplicity of algebras associated to \'etale groupoids, {\em Semigroup Forum} {\bf 88} (2014), 433--452.



\bibitem{Exel1} R.~Exel, Inverse semigroups and combinatorial $C^{\ast}$-algebras, {\em Bull. Braz. Math. Soc. (N.S.)} {\bf 39} (2008), 191--313.


\bibitem{Exel3} R.~exel, E.~Pardo, The tight groupoid of an inverse semigroup, {\em Semigroup Forum} {\bf 92} (2016), 274--303.


\bibitem{Fremlin} D. H. Fremlin, {\em Measure algebras, Volume three, Part II}, 2nd Edition, 2012, Torres-Fremlin, Printed by lulu.com.

\bibitem{FHS} J. Funk, P. Hofstra, B. Steinberg, Isotropy and crossed toposes, {\em TAC} {\bf 26} (2012), 660--709.

\bibitem{GH} S.~Givant, P.~Halmos, {\em Introduction to Boolean algebras}, Springer, 2009.




\bibitem{PJ} P. T. Johnstone, {\em Stone spaces}, CUP, 1982.

\bibitem{Kell1} J.~Kellendonk, The local structure of tilings and their integer group of coinvariants, {\it Commun. Math. Phys.} {\bf 187} (1997), 115--157.

\bibitem{Kell2} J.~Kellendonk, Topological equivalence of tilings, {\it  J. Math. Phys.} {\bf 38} (1997), 1823--1842.

\bibitem{Krieger} W. Krieger, On a dimension group for a class of homeomorphism groups, {\em Math. Ann.} {\bf 252} (1980), 87--95.


\bibitem{KL} G. Kudryavtseva, M. V. Lawson,  A perspective on non-commutative frame theory,  arXiv:1404.6516. 

\bibitem{KLLR} G. Kudryavtseva, M. V. Lawson, D. H. Lenz, P. Resende, Invariant means on Boolean inverse monoids,  {\em Semigroup Forum} {\bf 92} (2016), 77--101.

\bibitem{Kumjian} A. Kumjian, On localization and simple $C^{\ast}$-algebras, {\em Pacific J. Maths} {\bf 112} (1984), 141--192.


\bibitem{Lawson98} M.~V.~Lawson, {\em Inverse semigroups: the theory of partial symmetries}, World Scientific, 1998.


\bibitem{Lawson07b} M.~V.~Lawson, The polycyclic monoid $P_{n}$ and the Thompson groups $V_{n,1}$, {\em Comms in Algebra} {\bf 35} (2007), 4068--4087.

\bibitem{Lawson10a} M.~V.~Lawson, Compactable semilattices, {\em Semigroup Forum} {\bf 81} (2010), 187--199.

\bibitem{Lawson10b} M.~V.~Lawson, A non-commutative generalization of Stone Duality, {\em J. Aust. Math. Soc.} {\bf 88} (2010), 385--404.

\bibitem{Lawson12} M.~V.~Lawson, Non-commutative Stone duality: inverse semigroups, topological groupoids and $C^{\ast}$-algebras,  {\em IJAC} {\bf 22} (2012), 1250058, (47 pages).

\bibitem{LL} M.~V.~Lawson,  D.~H.~Lenz, Pseudogroups and their \'etale groupoids, {\em Adv. Math.} {\bf 244} (2013), 117--170.

\bibitem{LL2} M.~V.~Lawson,  D.~H.~Lenz, Distributive inverse semigroups and non-commutative Stone dualities, arXiv:1302.3032v1.

\bibitem{LMS} M.~V.~Lawson, S.~Margolis, B.~Steinberg, The \'etale groupoid of an inverse semigroup as a groupoid of filters,  {\em J. Austral. Math. Soc.} {\bf 94} (2014), 234--256.

\bibitem{LS} M.~V.~Lawson, P. Scott, AF inverse monoids and the structure of countable MV-algebras,  arXiv:1408.1231. 

\bibitem{Leech} J.~Leech, Inverse monoids with a natural semilattice ordering, {\em Proc. London Math. Soc.} (3) {\bf 70} (1995), 146--182.

\bibitem{Lenz} D.~H.~Lenz, On an order-based construction of a topological groupoid from an inverse semigroup, \emph{Proc. Edinb. Math. Soc.} {\bf 51} (2008), 387--406.

\bibitem{Matui12} H.~Matui, Homology and topological full groups of \'etale groupoids on totally disconnected spaces, 
{\em Proc. London Math. Soc.} (3) {\bf 104} (2012), 27--56.

\bibitem{Matui13} H.~Matui, Topological full groups of one-sided shifts of finite type, {\em J. Reine Angew. Math} {\bf 705} (2015), 35--84.


\bibitem{Munn} W. D. Munn, Congruence-free inverse semigroups, {\em Quart. J. Math. Oxford} (2) {\bf 25} (1974), 463--484.

\bibitem{Paterson} A.~L.~T.~Paterson, {\it  Groupoids, inverse semigroups, and their operator algebras}, Progress in Mathematics,  {\bf 170}, Birkh\"auser,  Boston, 1998.


\bibitem{Renault} J.~Renault, {\it A groupoid approach to $C^*$-algebras},  Lecture Notes in Mathematics,  {\bf 793}, Springer, 1980.

\bibitem{Renault2008} J.~Renault, Cartan subalgebras in $C^{\ast}$-algebras, {\em Irish Math. Soc. Bulletin} {\bf 61} (2008), 29--63.

\bibitem{Res1} P.~Resende, Lectures on \'etale groupoids, inverse semigroups and quantales, lecture notes for the GAMAP IP Meeting, Antwerp, 4-18 September, 2006, 115 pp.

\bibitem{Res2} P.~Resende, Etale groupoids and their quantales, {\em Adv. Math.} {\bf 208} (2007), 147--209.

\bibitem{Res3} P.~Resende, A note on infinitely distributive inverse semigroups, {\em Semigroup Forum} {\bf 73} (2006), 156--158.

\bibitem{Rubin} M. Rubin, On the reconstruction of topological spaces from their groups of homeomorphisms, {\em Trans Amer. Math. Soc.} {\bf 312} (1989), 487--538.



\bibitem{S2} B.~Steinberg, Simplicity, primitivity and semiprimitivity of \'etale groupoid algebras with applications to inverse semigroup algebras,
{\em J. Pure Appl. Algebra} {\bf 220} (2016), 1035--1054. 

\bibitem{W} F.~Wehrung, Refinement monoids, equidivisibility types, and Boolean inverse monoids, 205~pp, 2015, $<$hal-01197354$>$.

\end{thebibliography}
\end{document}